\documentclass[a4paper,12pt]{amsart}
\usepackage{amsfonts,amsthm,amsmath,amssymb,enumitem,etoolbox,comment,graphicx,array,tabu,commath,latexsym,environ,xcolor,tikz-cd,stackengine,diagbox,appendix,caption,mathrsfs,hyperref,cancel}
\usepackage[margin=1.2in]{geometry}
\usepackage{todonotes}
\hypersetup{
	colorlinks=true,
	linkcolor=blue,
	filecolor=blue,      
	citecolor=blue,
}

\usepackage{url}

\newcommand{\cC}{\mathcal{C}}

\newcommand{\cH}{\mathcal{H}}

\newcommand{\cK}{\mathcal{K}}

\newcommand{\cM}{\mathcal{M}}

\newcommand{\cS}{\mathcal{S}}

\newcommand{\ft}{\mathfrak{t}}

\newcommand{\R}{\mathbb R}

\newcommand{\Z}{\mathbb Z}

\newcommand{\genus}{\mathrm {genus}}

\newcommand{\Ent}{\mathrm {Ent}}

\newcommand{\length}{\mathrm{length}}

\newcommand{\dist}{\mathrm{dist}}

\renewcommand{\index}{\mathrm{index}}

\newcommand{\diam}{\mathrm{diam}}

\renewcommand{\tilde}{\widetilde}

\newcommand{\rank}{\mathrm{rank}}
\newcommand{\link}{\mathrm{link}}

\newcommand{\x}{\times}

\newcommand{\ins}{\mathrm{in}}
\newcommand{\out}{\mathrm{out}}
\newcommand{\sphere}{\mathrm{sphere}}
\newcommand{\neck}{\mathrm{neck}}

\newcommand{\loc}{\text{loc}}

\newtheorem{thm}{Theorem}[section]
\newtheorem{cor}[thm]{Corollary}
\newtheorem{prop}[thm]{Proposition}
\newtheorem{lem}[thm]{Lemma}
\newtheorem{conj}[thm]{Conjecture}

\theoremstyle{definition}
\newtheorem{defn}[thm]{Definition}
\newtheorem{exmp}[thm]{Example}
\newtheorem{rmk}[thm]{Remark}

\newtheorem{oqn}[thm]{Question}

\title{Genus one singularities in mean curvature flow}

\author{Adrian Chun-Pong Chu}
\author{Ao Sun}
\address{The University of Chicago, Department of Mathematics, Eckhart Hall,
	5734 S University Ave,
	Chicago, IL, 60637
	\\
	Current Address: Cornell University, Malott Hall, 212 Garden Ave, Ithaca, NY 14853}
\email{cc2938@cornell.edu}
\address{Lehigh University, Department of Mathematics, Chandler-Ullmann Hall, Bethlehem, PA 18015}
\email{aos223@lehigh.edu}

\begin{document}
	\maketitle

	\begin{abstract}
		We show that for certain one-parameter families of initial conditions in $\mathbb R^3$, when we run mean curvature flow, a genus one singularity must appear in one of the flows. Moreover, such a singularity is robust under perturbation of the family of initial conditions. This contrasts sharply with the case of just a single flow.
		As an application, we construct an embedded, genus one self-shrinker with entropy lower than a shrinking doughnut.
	\end{abstract}
	
	\section{Introduction}\label{sect_intro}
	
	Mean curvature flow (MCF) is the most rapid process to decrease the area of a surface. With an initial motivation from applied science, this geometric evolution equation has gained much interest recently due to its potential for studying the geometry and topology of surfaces embedded in three-manifolds. As a nonlinear geometric heat flow, MCF may have singularities, which may lead to changes in the geometry and topology of the surfaces.
	
	The blow-up method, pioneered by Huisken \cite{Huisken90}, Ilmanen \cite{Ilmanen95_Sing2D}, and White \cite{White97_Stratif}, shows that the singularities are modeled by a special class of surfaces called \emph{self-shrinkers}. They satisfy the equation $\vec{H}+\vec{x}^\perp/2=0$. 
	Determining the possible singularity models that can arise in an arbitrary MCF is a challenging problem. With the convexity assumption, Huisken \cite{Huisken86} proved that the singularities must be modeled by spheres. With the mean convexity assumption, White \cite{White97_Stratif, White00, White03} proved that the singularities must be modeled by spheres and cylinders. 
	However, in the absence of curvature assumptions, the question of which type of singularities must arise in MCF remains widely open. In this paper, we find a condition that guarantees the appearance of a singularity modeled by a genus one self-shrinker. To the best of our knowledge, this is the first resultthat produces a singularity, that appears in a non-self-shrinking flow and is modeled by a self-shrinker of non-zero genus.

	\begin{figure}[h]
		\centering
		\makebox[\textwidth][c]{\includegraphics[width=5.5in]{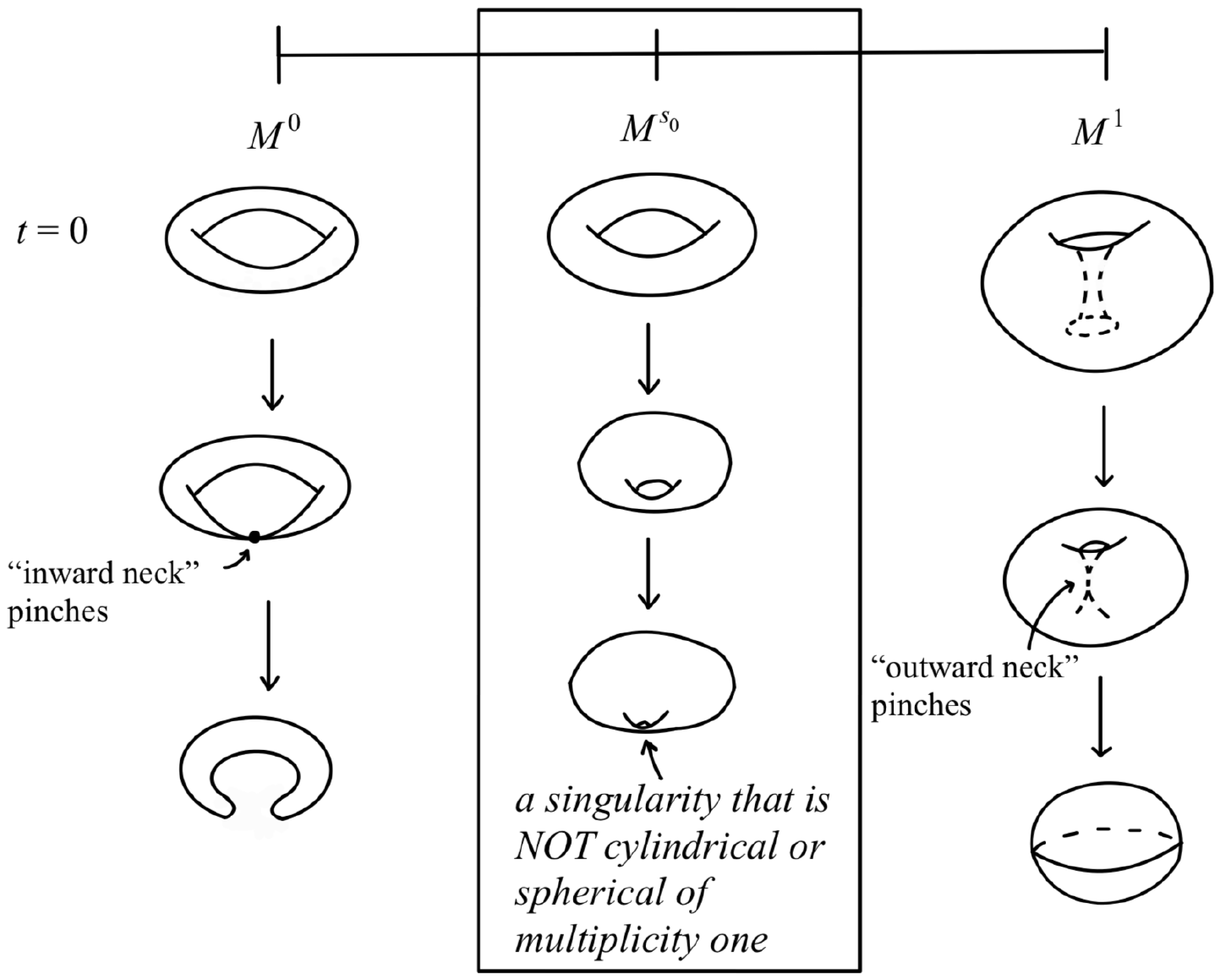}}
		\caption{}
		\label{fig:first_pic}
	\end{figure}

	Let us first explain the heuristics, which involves an interpolation argument. In Figure \ref{fig:first_pic}, we have a one-parameter family $\{M^s\}_{s\in[0,1]}$ of tori in the top row. Suppose that the initial torus $M^0$ has a thin ``inward neck,'' which will eventually pinch under the MCF. On the other hand, the final torus $M^1$ has a thin ``outward neck'' in the middle, which will also pinch under MCF. Then, there should exist a critical value $s_0\in[0,1]$ such that for the torus $M^{s_0}$,  both the inward and outward necks pinch under MCF, giving rise to a genus one singularity.

	The following is our main theorem. We will provide a precise definition of ``inward (or outward) torus neck will pinch" later in Definition \ref{defn_inward_neck_pinch}.

	\begin{thm}\label{thm:main1}
		Let $\{M^s\}_{s\in [0,1]}$ be a smooth family of tori in $\R^3$ such that for the MCF  starting from $M^0$ (resp. $M^1$), the inward (resp. outward) torus neck will pinch. Then there exists $s_0\in [0,1]$ such that the MCF starting from  $M^{s_0}$ would develop a singularity that is not multiplicity one cylindrical or multiplicity one spherical.
	\end{thm}
	
	Note that, in precise terms, by MCF we actually refer the level set flow (see \S \ref{sect_prelim}). In fact, before the flow encounter a genus one singularity, it is possible that it passed through some cylindrical singularities or spherical singularities.
	We also remark that Brendle \cite{Brendle16_genus0} proved that the only genus $0$ self-shrinkers are the spheres and the cylinders. In contrast, there are many higher genus self-shrinkers, as constructed in \cite{Angenent92_Doughnut, Nguyen14_Shrinker, KapouleasKleeneMoller18_DesingShrinker, Moller11_ClosedShrinker,  SunWangZhou20_MinmaxShrinker}, among others.
	
	Now, immediately, we can exclude the possibility of multiplicity if the \emph{entropy} of each torus $M^s$ is less than $2$. The entropy of a surface $\Sigma$ was defined by Colding-Minicozzi \cite{ColdingMinicozzi12_generic}:
	\[\Ent(\Sigma):=\sup_{x_0\in\R^3,t_0>0}(4\pi t_0)^{-1}\int_{\Sigma} e^{-\frac{|x-x_0|^2}{4t_0}}.\]
	
	\begin{cor}\label{cor_less_than_2}
		In the setting of Theorem \ref{thm:main1}, if each initial torus $M^s$ has entropy less than $2$, then at the singularity concerned, every tangent flow is given by a multiplicity one, embedded, genus one self-shrinker. 
	\end{cor}
	
	Recall that the tangent flow represents a specific blow-up limit of a MCF at a singularity, as discussed in \S \ref{sect_notation_setting}.
	By employing Huisken's monotonicity formula \cite{Huisken90}, Ilmanen \cite{Ilmanen95_Sing2D}, and White \cite{White97_Stratif} proved that the tangent flow must be a self-shrinker with multiplicity.
	
	Let us now explicitly provide a family of tori that satisfies the assumption of Corollary \ref{cor_less_than_2}. Consider the rotationally symmetric, compact, genus one self-shrinker in $\mathbb{R}^3$ constructed by Drugan-Nguyen \cite{DruganNguyenShrinkingDonut}, which we will denote by $\mathbb{T}$. It is worth noting that both $\mathbb{T}$ and the Angenent torus \cite{Angenent92_Doughnut} are referred to as {\it shrinking doughnuts}, and they may be the same. It was shown that $\mathbb{T}$ has entropy strictly less than $2$ \cite{DruganNguyenShrinkingDonut}, while Berchenko-Kogan \cite{Berchenko-Kogan21_Angenenet_torus_entropy} provided numerical evidence that the Angenent torus has an entropy of approximately $1.85$.

	\begin{thm}\label{thm:main2}
		Let $\{M^s\}_{s\in [0,1]}$ be a smooth family of tori in $\R^3$ that are sufficiently close in $C^\infty$ to the shrinking doughnut $\mathbb T$, with $M^0$ strictly inside $\mathbb T$ while $M^1$ strictly outside. Then there exists $s_0\in [0,1]$ such that the MCF starting from  $M^{s_0}$ would develop a singularity at which every tangent flow is given by a multiplicity one, embedded, genus one self-shrinker.
	\end{thm}

	The idea of Theorem \ref{thm:main2} can be traced back to the work of Lin and the second author in \cite{Lin-Sun2022_bifurcation}. In earlier work, Colding-Ilmanen-Minicozzi-White \cite{CIMW13_EntropyMinmzer} observed that one can perturb a closed embedded self-shrinker in $\R^3$ such that the MCF has only neck and spherical singularities. Lin and the second author observed a bifurcation phenomenon: Inward (resp. outward) perturbations cause the MCF pinch from inside (resp. outside). After we completed this manuscript, we were notified by the anonymous referee that the idea of Theorem \ref{thm:main1} has been discussed and explained orally by Edelen and White.
	
	It is also interesting to compare our results with the recent developments in generic MCF \cite{ColdingMinicozzi12_generic, CCMS20_GenericMCF, CCMS21_GenericMCF_LowEntropy, SunXue2021_initial_conical, SunXue2021_initial_closed, chodoshchoischulze2023mean, sun-generic-multi-1}: One can perturb a single MCF to avoid a singularity that is not spherical or cylindrical. In contrast, our results imply that for a certain one-parameter family of MCFs, a singularity that is modeled by a genus one shrinker remains robust under perturbations.

	It is natural to ask whether Theorem \ref{thm:main1} extends to surfaces with genus two or above. Actually, it would not: See a counterexample in Remark \ref{rmk_main_thm_fail}. Nevertheless, a similar theory might be established for a multi-parameter family of higher genus surfaces (see Question \ref{conj_higher_genus}).
	
	Let us now present several applications of the above theorems.
	\begin{thm}\label{thm_genus1_least_entropy}
		An embedded, genus one self-shrinker in $\R^3$ of the least entropy either is non-compact or has index $5$.
	\end{thm}
	
	\begin{rmk}
		In \cite{BuzanoNguyenSchulz+2025+35+52}, Buzano-Nguyen-Schulz used equivariant min-max method to construct a number of noncompact shrinkers with nontrivial topology and symmetry. In particular, they constructed a genus $1$ shrinker with dihedral symmetry, which is the first example of a non-compact genus one shrinker. They also used numerical simulation to show that such a genus one shrinker has Gaussian area less than the Angenent torus, and they conjectured that the index of such a shrinker is $5$. Combining with our Theorem \ref{thm_genus1_least_entropy}, it is plausible  that  Buzano-Nguyen-Schulz's example  has the least entropy among the genus one shrinkers.
	\end{rmk}
	
	Note that the existence of an entropy minimizer among all embedded, genus $g$ self-shrinkers in $\R^3$, with a fixed $g$, was proved by Sun-Wang \cite{SunWang2020compactness}.
	
	\begin{thm}\label{thm_eternal}
		There exists an ancient MCF through cylindrical and spherical singularities $\{M(t)\}_{t<0}$ in $\R^3$ such that:
		\begin{itemize}
			\item As $t\to-\infty$, $\frac{1}{\sqrt{-t}} M(t)\to \mathbb T$ smoothly.
			\item As $t\to 0$, $M(t)$ hits a singularity at which every tangent flow is given by a multiplicity one, embedded, genus one self-shrinker of lower entropy than $\mathbb T$.
		\end{itemize}
	\end{thm}
	
	In fact, Theorem \ref{thm_eternal} remains valid even with $\mathbb{T}$ replaced by any other closed, embedded, rotationally symmetric, genus one shrinker (if they indeed exist), and the same proof will hold.
	
	Recalling that the rotationally symmetric shrinker $\mathbb{T}$ must have index of at least $7$, as shown by Liu \cite{Liu2016_index_shrinker}, we can deduce the following corollary from Theorem \ref{thm_genus1_least_entropy} and \ref{thm_eternal}.
	\begin{cor}
		There exists an embedded, genus one self-shrinker in $\R^3$ with entropy lower than $\mathbb T$.
	\end{cor}

	Finally, the three self-shrinkers in $\R^3$ with the lowest entropy are the plane, the sphere, and the cylinder (\cite{CIMW13_EntropyMinmzer, BernsteinWang17_small_entropy}). Notably, all three of them are rotationally symmetric. Kleene-M\o ller \cite{KM_rotational} proved that all other rotationally symmetric smooth embedded self-shrinkers are closed with genus $1$. 
	
	Now,
	the space of smooth embedded self-shrinkers in $\R^3$ with entropy less than some constant $\delta<2$ is known to be compact in the $C_\text{loc}^\infty$ topology (see \cite{lee2021compactness}). Together with the rigidity of the cylinder as a self-shrinker by \cite{ColdingIlmanenMinicozzi15}, there exists a smooth embedded self-shrinker that minimizes entropy among all smooth embedded self-shrinkers with entropy larger than that of the cylinder.
	
	\begin{cor}\label{cor_fourth}
		A smooth embedded self-shrinker in $\R^3$ with the fourth lowest entropy is not rotationally symmetric.
	\end{cor}

	\subsection{Main ideas: Change in homology under MCF}\label{sect_topo_mcf}

	The major challenge of this paper is to introduce some new concepts to rigorously state and prove the interpolation argument we outlined on page 1 and Figure \ref{fig:first_pic}. Particularly, it is crucial to describe the topological change of the surfaces more precisely. Let $\cM=\{M(t)\}_{t\geq 0}$ be a MCF in $\R^3$, where the initial condition $M(0)$ is a closed, smooth, embedded surface. Since we would allow $M(t)$ to have singularities and thus change its topology,  $\cM$ is, more precisely, a \emph{level set flow}. {\it In this paper, we often use the phrases MCF and level set flow interchangeably.}
	
	It is known that the topology of $M(t)$ simplifies over time. In \cite{White95_WSF_Top}, White focused on describing the complement $\R^3\backslash M(t)$ (instead of $M(t)$ itself), and how it changes over time. For example, he showed $\rank (H_1(\R^3\backslash M(t)))$ is non-increasing in $t$, where $H_1$ denotes the first homology group in $\Z$-coefficients. Therefore, heuristically, the topology can only be destroyed but not created during the evolution of the surface.
	
	In this paper, we will further describe this phenomenon by {\it keeping track of which elements of the initial homology group $H_1(\R^3\backslash M(0))$ are destroyed, and how they are destroyed.} To illustrate, let us use the flow depicted in Figure \ref{fig:genus2} as an example.
	
	\begin{figure}[h]
		\centering
		\makebox[\textwidth][c]{\includegraphics[width=6in]{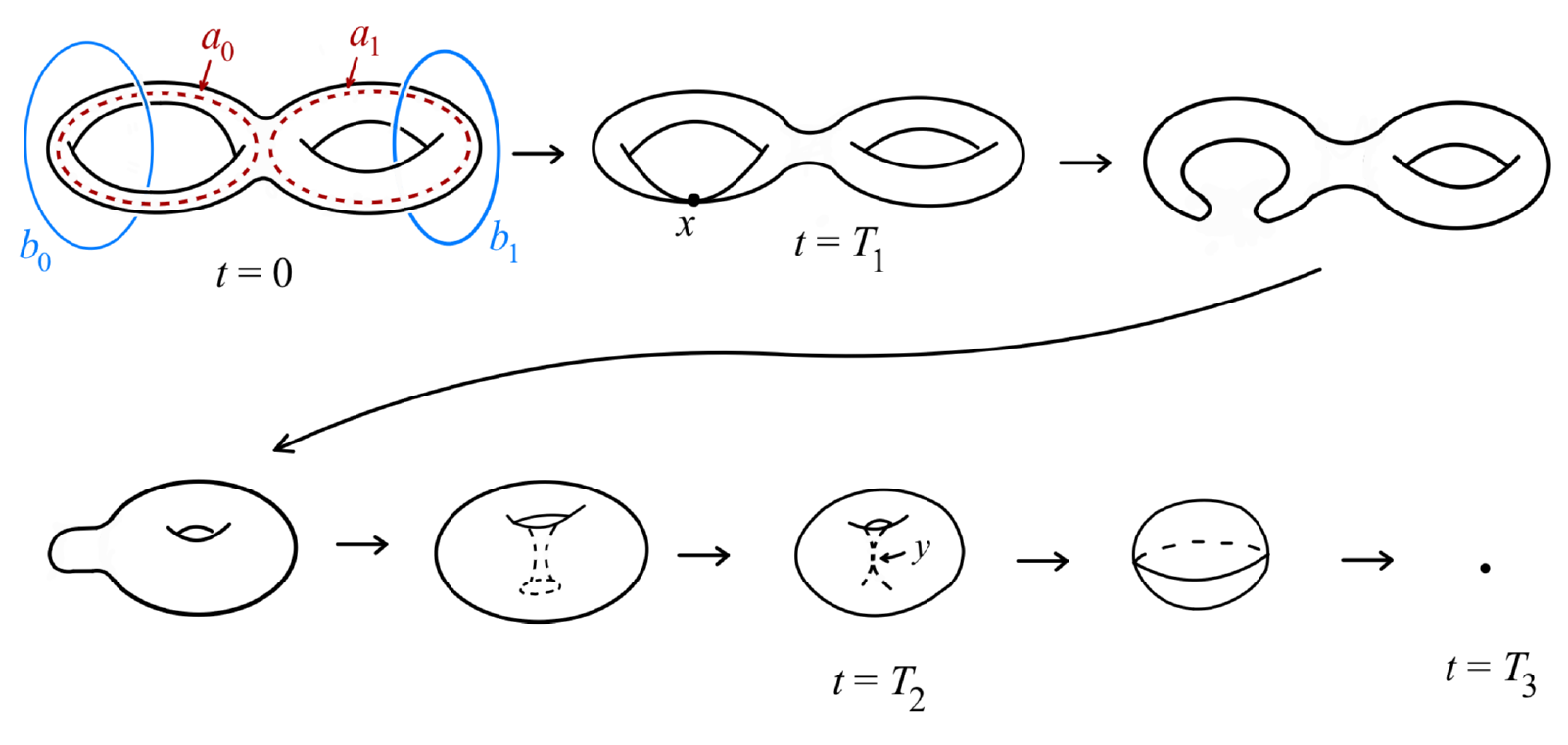}}
		\caption{}
		\label{fig:genus2}
	\end{figure}
	
	\subsubsection{Heuristic observation}
	Let us begin by providing some heuristic observations regarding Figure \ref{fig:genus2}. We will elaborate on them more precisely shortly. We fix four elements of $H_1(\R^3\backslash M(0))$ at time $t=0$, as shown in the figure. Note that $a_0$ and $a_1$ are in the bounded region {\it inside} the genus two surface $M(0)$, whereas $b_0$ and $b_1$ are in the region {\it outside} $M(0)$.
	\begin{enumerate}
		\item At time $t=T_1$, $a_0$ is ``broken" by the cylindrical singularity $x$  of the flow. As a result, for later time $t>T_1$, $a_0$ no longer exists. Apparently, it ``terminates'' at time $T_1$.
		\item\label{item_heuristic_survive} On the other hand, $a_1$, $b_0$, and $b_1$ all can survive through time $T_1$. For example, for $b_0$, we can clearly have a {\it continuous} family of  loops, $\{\beta_t\}_{t\geq 0}$, where $[\beta_0]=b_0$ and  each $\beta_t$ is a loop {\it outside} the surface $M(t)$. In this sense, $b_0$ will survive for all time, although it becomes {\it trivial} after time $T_1$.
		\item As for $b_1$, although it survives through $t=T_1$, it will terminate at $t=T_2$, when it is broken by the cylindrical singularity $y$.
	\end{enumerate}
	Let us now provide precise descriptions of these observations.
	
	\subsubsection{Three new concepts} 
	To our knowledge, these concepts are new, but they seem natural in the context of geometric flows. We believe these concepts may hold independent interest as well.
	
	To set up, for any two times $t_1<t_2$, let us  consider the {\it complement} of the spacetime track of the flow within the time interval $[t_1,t_2]$:
	$$W[t_1,t_2]:=\bigcup_{t\in [t_1,t_2]}(\R^3\backslash M(t))\x\{t\}\subset \R^3\x [t_1,t_2].$$
	In order to discuss the ``termination'' of an element $c_0\in H_1(\R^3\backslash M(0))$ under the flow, we first need to \emph{relate  elements of $H_1(\R^3\backslash M(0))$ and elements of $H_1(\R^3\backslash M(t))$} at some later time $t>0$.
	
	\begin{quote}
		{\bf Homology descent.} (Definition \ref{defn_order}.) Given two elements $c_0 \in H_1(\R^3\backslash M(0))$ and   $c \in H_1(\R^3\backslash M(t))$ with $t>0$, we say that  {\it $c$ descends from $c_0$}, and denote 
		$$c_0\succ c,$$
		if the following holds: 
		For every representative $\gamma_0\in c_0$ and $\gamma\in c$, if we view them as subsets
		$$\gamma_0\subset (\R^3\backslash M(0))\x\{0\},\;\; \gamma\subset (\R^3\backslash M(t))\x\{t\},$$
		then they bound some singular 2-chain $\Gamma\subset W[0,t]$, i.e. 
		$\gamma_0-\gamma=\partial\Gamma.$ (See Figure \ref{fig::descent}.)
	\end{quote}
	\begin{figure}[h]
		\centering
		\makebox[\textwidth][c]{\includegraphics[width=3.5in]{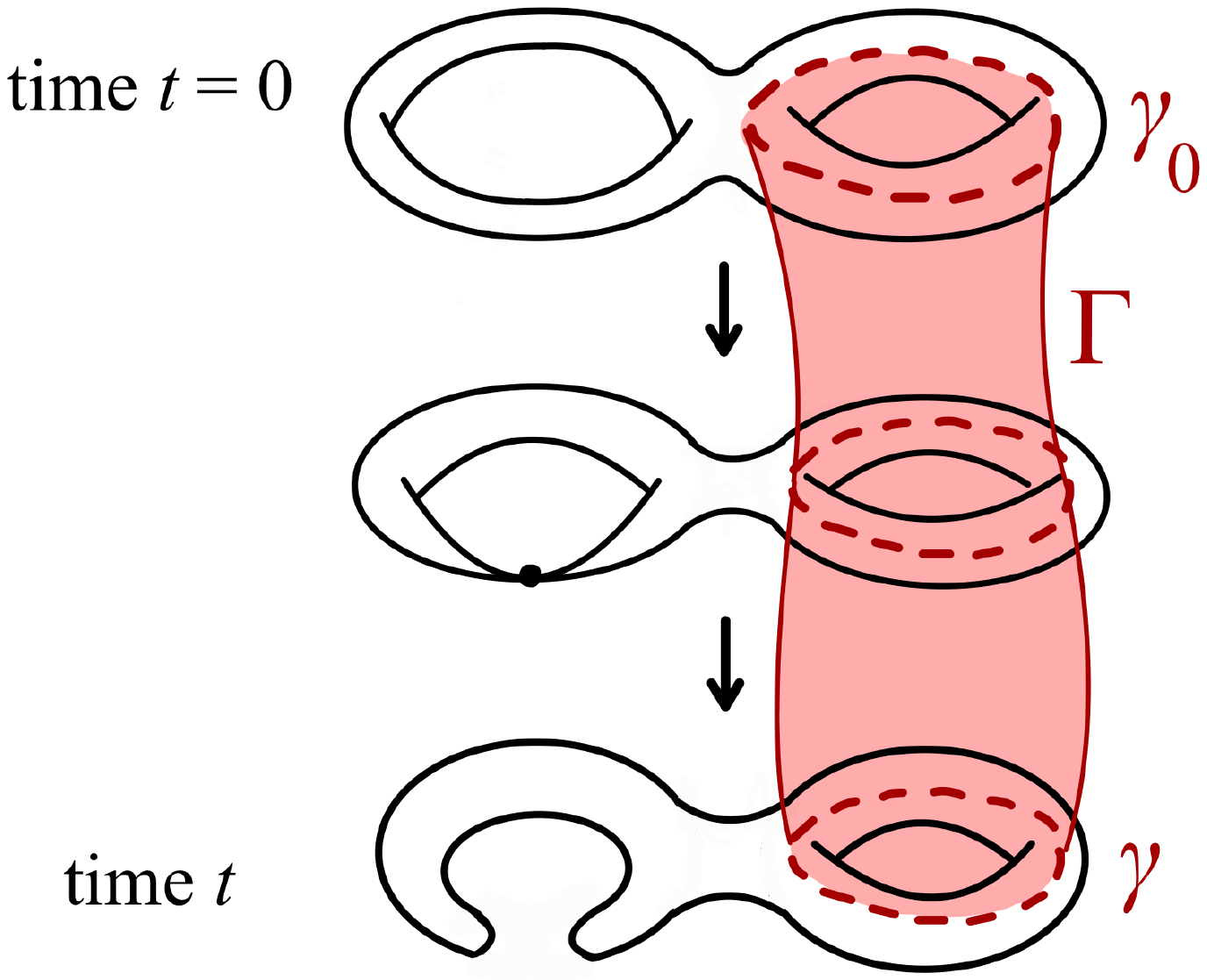}}
		\caption{}
		\label{fig::descent}
	\end{figure}
	
	As we will prove, the above notion satisfies some desirable properties. For example, given a $c_0 \in H_1(\R^3\backslash M(0))$, the element $c \in H_1(\R^3\backslash M(t))$  described above, if exists, turns out to be {\it unique}. Consequently, we denote this unique element as $c_0(t)$. 
	
	This enables us to further define:
	\begin{quote}
		{\bf Homology termination.} (Definition \ref{defn_termination}.) Let $c_0\in H_1(\R^3\backslash M(0))$. If
		$$\mathfrak{t}(c_0):=\sup\{t\geq 0:c_0\succ c \textrm{ for some }c\in H_{1}(\R^3\backslash M(t))\}$$
		is finite, then we say that $c_0$ {\it terminates at time} $\mathfrak{t}(c_0)$.
	\end{quote}
	For instance, in Figure \ref{fig:genus2}, we observe that $a_0$ terminates at time $T_1$, and $b_1$ terminates at time $T_2$. However, $b_0$ \emph{never} terminates, despite the fact that $b_0(t)$ becomes trivial for $t>T_1$. Similarly, $a_1$ also \emph{never} terminates, even though $a_1(t)$ becomes trivial for $t>T_2$. Note that $a_1$ would not terminate at time $T_3$: For any $t>T_3$, any loop in $\mathbb R^3\backslash M(t)=\R^3$  would bound a disc in $\mathbb R^3$, so it follows easily that for any loop $\gamma_0\in a_1$ and  loop $\gamma\subset \R^3\times\{t\}$, $\gamma_0-\gamma$ would bound some 2-dimensional chain in the complement of the spacetime track.
	
	Finally, we can describe what ``$a_0$ breaks at a cylindrical singularity $x$" means.
	\begin{quote}
		{\bf Homology breakage.} (Definition \ref{defn_breakage}.)
		Let $c_0\in H_1(\R^3\backslash M(0))$,  $T> 0$, and $x\in M(T)$. Suppose  the following holds:
		\begin{itemize}
			\item For each $t\in [0,T)$, {\it the} element $c_0(t)\in H_1(\R^3\backslash M(t))$ (such that $c_0\succ c_0(t)$) exists.
			\item For every neighborhood $U\subset\R^3$ of $x$, for each $t<T$ sufficiently  close  to $T$, every element of $c_0(t)$ intersects $U$.
		\end{itemize}
		Then we say that $c_0$ {\it breaks at}  $(x,T)$. (See Figure \ref{fig::breakage}.)
	\end{quote}
	\begin{figure}[h]
		\centering
		\makebox[\textwidth][c]{\includegraphics[width=1.5in]{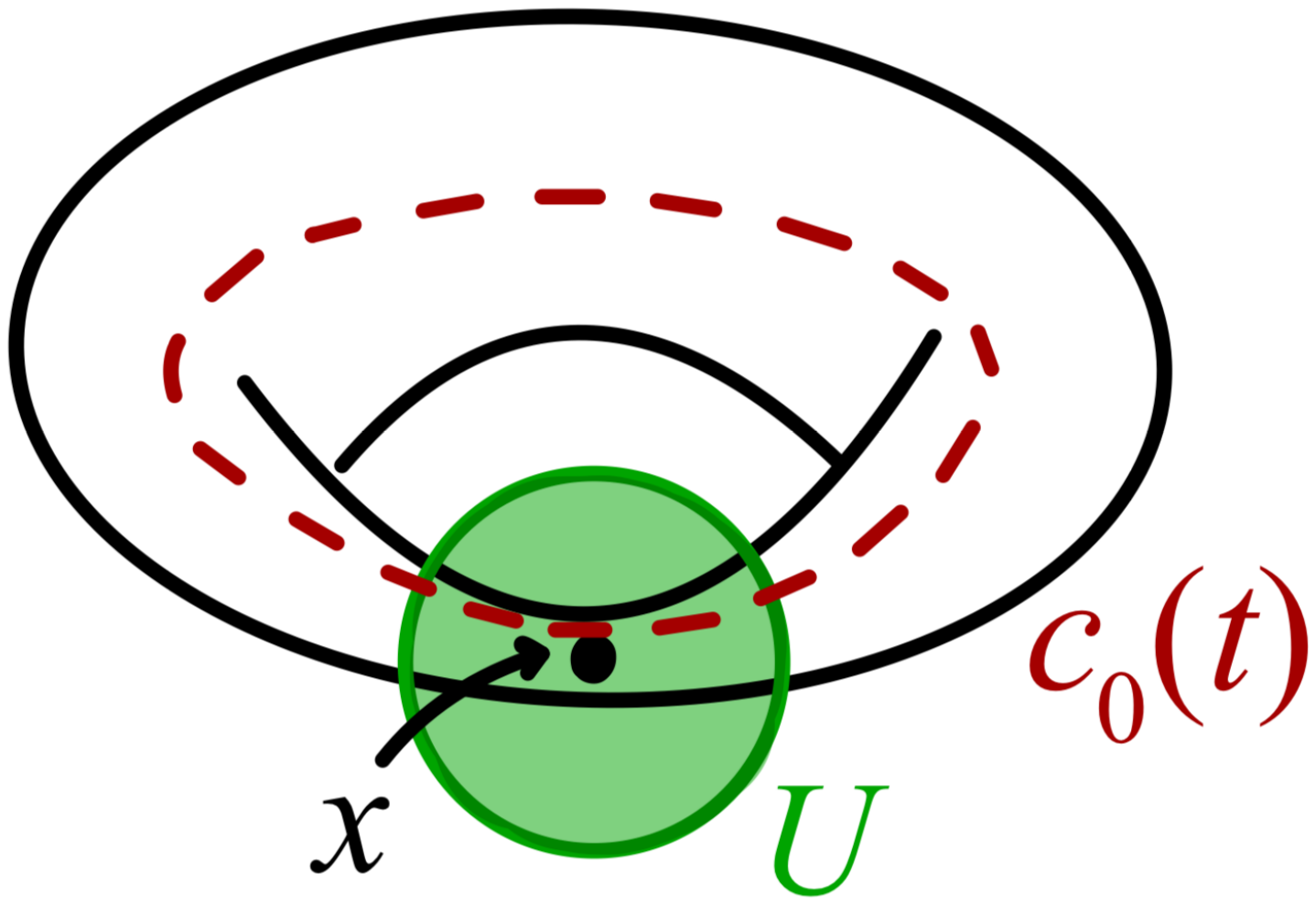}}
		\caption{The picture at time $t$, for all $t<T$ sufficiently close to $T$.}
		\label{fig::breakage}
	\end{figure}
	
	For example, in Figure \ref{fig:genus2}, $a_0$ breaks at $(x,T_1)$, while $b_1$ breaks at $(y,T_2)$.
	
	As we will see, these three new concepts are quite useful and satisfy several nice properties. Here are a few examples:
	\begin{itemize}
		\item A homology class cannot break at a regular point, nor a spherical singularity of the flow (Proposition \ref{prop_no_reg} and \ref{prop_no_sphere}).
		\item If the initial condition $M(0)$ is a closed surface of non-zero genus, then some initial homology class must terminate at finite time (Remark \ref{rmk_any_genus}).
		\item Suppose $\{M(t)\}_{t\geq 0}$ is a MCF with only spherical and cylindrical singularities. If a homology class terminates at some time $T$, then it must break at $(x,T)$ for some cylindrical singularity $x\in M(T)$ (Theorem \ref{thm_exist_break_point}).
	\end{itemize}
	These properties are all crucial in proving the main theorems.
	
	Finally, let us provide a precise definition of  ``inward (or outward) torus neck will pinch'' in Theorem \ref{thm:main1}.
	\begin{defn}\label{defn_inward_neck_pinch}
		Given a torus $M$ in $\R^3$, let $a_0$ (resp. $b_0$) be a generator of the first homology group of the interior (resp. exterior) region of $M$, which is isomorphic to $\Z$ (see Figure \ref{fig:a0b0}).    
		We say that the {\it inward (resp. outward) torus neck of $M$ will pinch} if $a_0$ (resp. $b_0$) will terminate under MCF.
	\end{defn}
	\begin{figure}[h]
		\centering
		\makebox[\textwidth][c]{\includegraphics[width=1.7in]{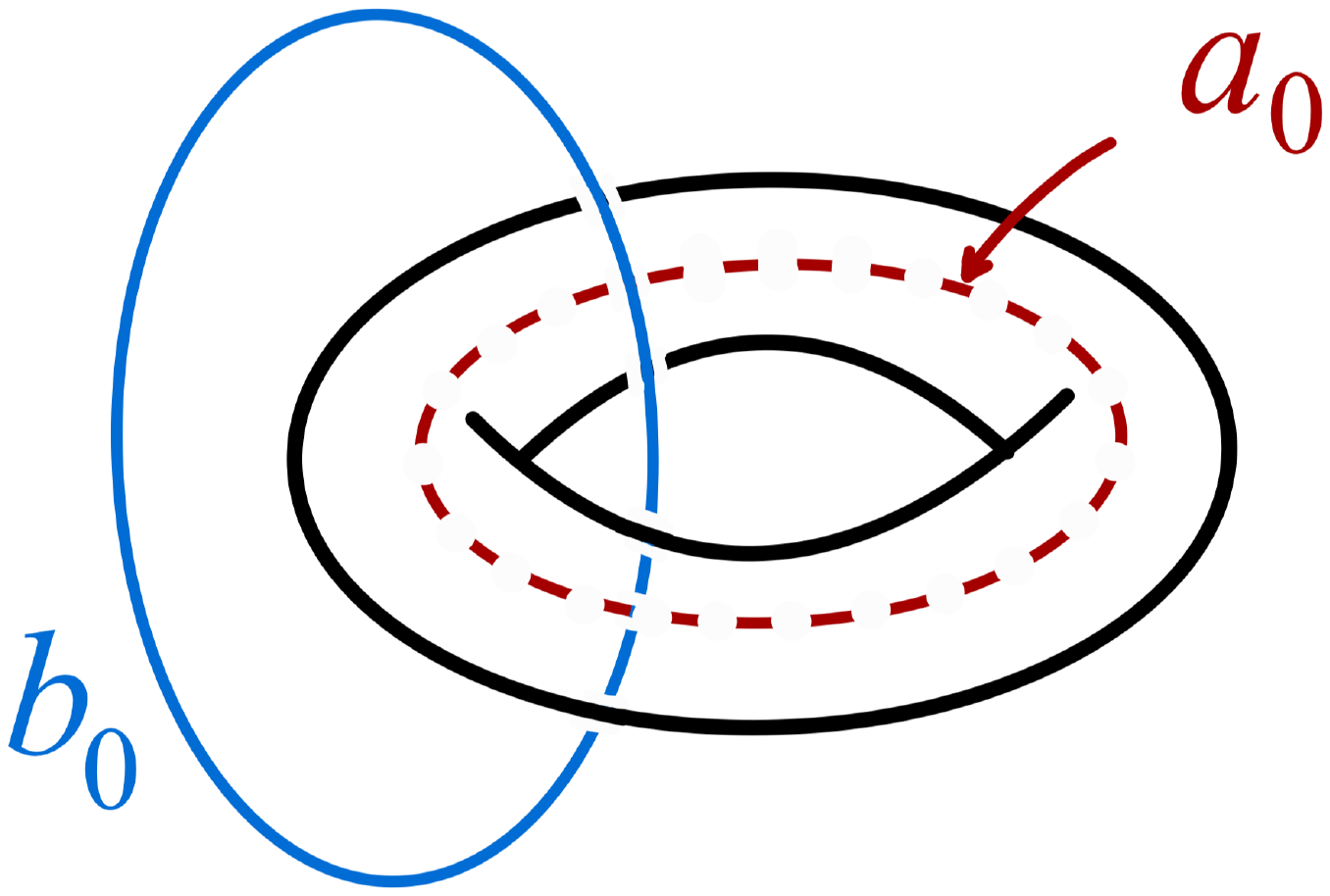}}
		\caption{}
		\label{fig:a0b0}
	\end{figure}
	
	Clearly, $a_0$ (and $b_0$) is unique up to a sign, and the above notion is independent of which sign we choose.

	\subsection{Structure of cylindrical singularities} Once we establish the topological concepts to keep track of the homology classes under the MCF, another challenge arises: We need to understand what happens to these homology classes as the MCF encounters the cylindrical singularities.
	
	Intuitively, a cylindrical singularity is just like a neck, and as we approach the singular time, the neck pinches as in Figure \ref{fig:first_pic}. However, the actual situation can be much more complicated. For example, consider the MCF of the boundary of a tubular neighborhood of a rotationally symmetric $S^1$ in $\R^3$. It will shrink to a singular set that is a rotationally symmetric $S^1$, where each singular point is cylindrical, but it does not look like a neck pinching.
	
	First, one has the partial regularity of the singular set of cylindrical singularities, studied by White \cite{White97_Stratif} and Colding-Minicozzi \cite{CM15_Lojasiewicz, ColdingMinicozzi16_SingularSet}. This allows us to control the singular set. We can establish the compactness of the singular set of cylindrical singularities that are inward (or outward), and know that they only appear for a zero-measured set of time. 
	
	Another important theory is the mean convex neighborhood theory of cylindrical singularities by Choi-Haslhofer-Hershkovits \cite{ChoiHaslhoferHershkovits18_MeanConvNeighb}, and a generalized version by Choi-Haslhofer-Hershkovits-White \cite{ChoiHaslhoferHershkovitsWhite22_AncientMCF}. In these works, they classified the possible limit flows at a cylindrical singularity. As a consequence, they derived a canonical neighborhood theorem at a cylindrical singularity, which describes the local behavior of the MCF.
	
	We will study the local behavior of MCF at cylindrical singularities based on these two theories. Nevertheless, the particular local behavior we need to understand does not directly come from \cite{ChoiHaslhoferHershkovits18_MeanConvNeighb, ChoiHaslhoferHershkovitsWhite22_AncientMCF}. We present these relevant results in \S\ref{SS:MCF through cylindrical and spherical singularities}.

	\subsection{Outline of proofs}
	\subsubsection{Theorem \ref{thm:main1} } 
	We will prove by contradiction. 
	For each $s\in [0,1]$, let $\cM^s=\{M^s(t)\}_{t\geq 0}$ be the MCF (more precisely, a level set flow) with $M^s(0)=M^s$ as its initial condition. Let $a_0$ (resp. $b_0)$ be a generator of the first homology group of the inside (resp. outside) region of each torus $M^s$ (recall Definition  \ref{defn_inward_neck_pinch}). Assuming that Theorem  \ref{thm:main1} were false, $\cM^s$ would be a MCF through cylindrical and spherical singularities for each $s$. This flow is unique and well-defined by  Choi-Haslhofer-Hershkovits \cite{ChoiHaslhoferHershkovits18_MeanConvNeighb}. Next, we show that for each $s$, either $a_0$ or $b_0$ will terminate, but {\it not both}. This claim relies on the fact, mentioned above, that if a homology class will terminate, it must break at a neck singularity. This crucial fact is established based on the mean convex neighborhood theorem and the canonical neighborhood theorem by Choi-Haslhofer-Hershkovits-White \cite{ChoiHaslhoferHershkovits18_MeanConvNeighb, ChoiHaslhoferHershkovitsWhite22_AncientMCF}. 
	
	Thus, we can partition $[0,1]$ into a disjoint union $A\sqcup B$, where $A$ is the set of $s$ for which $a_0$ will terminate, and $B$ is the set of $s$ for which $b_0$ will terminate. Furthermore, we will show that $A$ and $B$ are both closed sets. Recall that we are given $0\in A$ and $1\in B$. Since $[0,1]$ is a connected interval, this leads to a contradiction.
	
	\subsubsection{Theorem \ref{thm:main2}} 
	We can apply Theorem \ref{thm:main1} to prove Theorem \ref{thm:main2}, provided that we can show the inward torus neck will pinch (i.e., $a_0$ will terminate) for the starting flow ($s=0$), and the outward torus neck will pinch (i.e., $b_0$ will terminate) for the ending flow ($s=1$).
	To prove, for instance, that $a_0$ will terminate for the starting flow, we recall that $M^0(0)$ lies strictly inside the shrinker  $\Sigma$. Then we will run MCF to these two surfaces and \emph{use the avoidance principle, which states that the distance between the two surfaces will increase, to conclude that $a_0$ must terminate.}
	
	\subsubsection{Theorem \ref{thm_genus1_least_entropy}}
	Let $\Sigma$ be an embedded, genus one shrinker with the least entropy. Suppose by contradiction that it is compact with index at least $6$. 
	Disregarding the four (orthogonal) deformations induced by translation and scaling, there are still two other deformations that decrease the entropy, one of which is the one-sided deformation given by the first eigenfunction of the Jacobi operator. Thus, we can construct a one-parameter family of tori with entropy less than $\Sigma$, such that the starting torus is inside $\Sigma$, and the ending torus is outside $\Sigma$.
	Then, as in the proof of Theorem \ref{thm:main2}, we apply Theorem \ref{thm:main1} to obtain \emph{another} genus one shrinker with less entropy than $\Sigma$. This contradicts the definition of $\Sigma$.
	
	\subsubsection{Theorem \ref{thm_eternal}}
	According to Liu \cite{Liu2016_index_shrinker}, the shrinking doughnut $\mathbb{T}$ has an index of at least $7$. Consequently, based on the result of Choi-Mantoulidis \cite{ChoiMantoulidis22AncientGradientFlows}, there exists a one-parameter family of ancient \emph{rescaled} MCF originating from $\mathbb T$ that decreases the entropy. As before, we can apply Theorem \ref{thm:main1} to immediately obtain the desired genus one, self-shrinking tangent flow with lower entropy.
	
	\subsection{Open questions}
	
	We propose several open problems. The first one is motivated by generic MCF and min-max theory. 
	
	\begin{conj}
		There exists an embedded, genus one, index $5$ self-shrinker in $\R^3$ that is the  ``second most generic" one.
	\end{conj}
	
	We say a self-shrinker $\Sigma$ is the ``second most generic", after the generic ones (the cylinder and the sphere), in the following sense:
	Suppose we have a one-parameter family of embedded surfaces $\{M^s\}_{s\in [0,1]}$ in $\R^3$. Then, we can perturb {\it this family} such that when we run MCF for every $M^s$, every singularity is either cylindrical, spherical, or modeled by $\Sigma$. 
	
	Note that Theorem \ref{thm_genus1_least_entropy} and its proof can be seen as evidence of a very ``local'' version of this conjecture: They say that any closed, embedded, genus one self-shrinker with an index of at least $6$ is not the second most generic.
	
	Now, we note that  Theorem \ref{thm:main1} does not hold for initial conditions with genus greater than one, see Remark \ref{rmk_main_thm_fail}. 
	
	\begin{oqn}\label{conj_higher_genus}
		Can  Theorem \ref{thm:main1} be generalized to the higher genus case, possibly by considering higher parameter families of initial conditions?
	\end{oqn}
	
	Finally, notice that many concepts that we introduce in this paper heavily rely on the extrinsic structure of mean curvature flow.
	
	\begin{oqn}
		Can the concepts of homology descent, homology termination, and homology breakage be adapted to the setting of Ricci flow?
	\end{oqn}
	
	\subsection{Organizations.} In \S \ref{sect_prelim}, we will introduce the preliminary materials, including a refined canonical neighborhood theorem. In \S \ref{sect_general_results}, we will define the concepts of homology descent, homology termination, and homology breakage, and prove some relevant basic propositions. In \S \ref{sect_different_termination_time}, we focus on the case of MCF through cylindrical and spherical singularities, with torus as the initial condition.
	In \S \ref{sect_proof}, we prove the main theorems.
	
	\subsection*{Acknowledgement} We would like to thank Professor Andr\'e Neves for all the fruitful discussions and his constant support. We are grateful to Zhihan Wang for the valuable conversations. And the first author would also like to thank Chi Cheuk Tsang for the helpful discussions. We are also grateful to anonymous referees for many helpful comments and suggestions, especially the work by Edelen and White. We want to thank Professor Reto Buzano for bringing our attention to their work \cite{BuzanoNguyenSchulz+2025+35+52}.

	\section{Preliminaries}\label{sect_prelim}
	In \S \ref{sect_prelim} we will set up the language and provide the necessary background to define MCF through cylindrical and spherical singularities. 
	
	The classical \emph{mean curvature flow} is a family of hypersurfaces $\{M(t)\}_{t\in[0,T)}$ in $\R^{n+1}$ satisfying the equation
	\begin{equation}
		\partial_t x=\vec H(x),
	\end{equation}
	where $x$ is the position vector and $\vec H$ is the mean curvature vector. When the hypersurface is not $C^2$, we can not define the mean curvature flow using this PDE, and we need to use some weak notions to define the flow.
	\subsection{Weak solutions of MCF}
	Throughout this paper, we will focus on two different types of weak solutions of MCF. One is a set-theoretic weak solution defined by the {\it level set flow}, and another one is a geometric measure theoretic weak solution called {\it Brakke flow}. Readers interested in detailed discussions of level set flows can refer to \cite{EvansSpruck91, Ilmanen92_LSF}, while those interested in Brakke flow can refer to \cite{Brakke78, Ilmanen94_EllipReg}.
	
	The level set flow equation is a degenerate parabolic equation
	\begin{equation}\label{eq:LSF}
		\partial_t u=\Delta u-\left(\frac{D^2u(Du,Du)}{|Du|^2}\right).
	\end{equation}
	Suppose $M(0)$ is a closed hypersurface in $\R^{n+1}$, then if $u(\cdot,t)$ solves \eqref{eq:LSF} with $M(0)=\{x\in\R^{n+1}:u(\cdot,0)=0\}$, then $M(t):=\{x\in\R^{n+1}:u(\cdot,0)=0\}$ can be viewed as a weak solution to MCF. In particular, when $M(t)$ is smooth, this weak solution coincides with the classical solution of MCF.
	
	The level set flow was introduced by Osher-Sethian in \cite{OsherSethian88}. Chen-Giga-Goto \cite{ChenGigaGoto91_LSF} and Evans-Spruck \cite{EvansSpruck91} introduced the viscosity solutions to equation \eqref{eq:LSF}, and these solutions are Lipschitz. Throughout this paper, when we refer to a {\it level set function} or a solution to the level set flow equation, we mean a viscosity solution to equation \eqref{eq:LSF}.
	
	The set-theoretic solution of a MCF will be called the {\it level set flow} or {\it biggest flow}. These notions are used by Ilmanen \cite{Ilmanen92_LSF} and White \cite{White95_WSF_Top, White00, White03}. The term ``biggest flow'' is used to avoid ambiguity when dealing with weak solutions for noncompact flows. 
	Such a weak solution may have a nonempty interior. In this case, we say the level set flow \emph{fattens}.
	
	\bigskip
	
	Brakke flow is defined using geometric measure theory. Let $X$ be a complete manifold without boundary. The Brakke flow is a family of Radon measures $\{\mu_t\}_{t\geq 0}$, such that for any test function $\phi\in C_c^2(X)$ with $\phi\geq 0$,
	\[
	\limsup_{s\to t}\frac{\mu_s(\phi)-\mu_t(\phi)}{s-t}\leq \int(-\phi H^2+\nabla^\perp\cdot\vec{H})d\mu_t,
	\]
	where $\vec{H}$ is the mean curvature vector of $\mu_t$ whenever $\mu_t$ is rectifiable and has $L^2$-mean curvature in the varifold sense. Otherwise, the right-hand side is defined to be $-\infty$. 
	
	In general, the Brakke flow starting from a given initial data is not unique. We will be interested in unit regular cyclic integral Brakke flows. For detailed discussions on these notions, we refer the readers to \cite{White09_CurrentsVarifolds}. The existence of such a flow starting from a smooth surface is guaranteed by Ilmanen's elliptic regularization, see \cite{Ilmanen94_EllipReg}. These flows have a well-established compactness theory.
	
	\subsection{Setting and notations}\label{sect_notation_setting}
	Let $M(0)$ be a closed smooth $n$-dimensional hypersurface in $\R^{n+1}$ that bounds a compact set $K_\ins(0)$. Let $K_\out(0)=\overline{\R^{n+1}\backslash K_\ins(0)}$. Now, denote by $$\{M(t)\}_{t\geq 0},\{K_\ins(t)\}_{t\geq 0},\textrm{ and }\{K_\out(t)\}_{t\geq 0}$$  respectively the level set flow (i.e. the biggest flow)
	with initial condition $M(0),K_\ins(0)$, and $K_\out(0)$. Then we define their {\it spacetime tracks}
	\begin{align*}
		\mathcal M&=\{(x,t):x\in M(t),t\geq 0\},\\
		\cK_\ins&= \{(x,t):x\in K_\ins(t),t\geq 0\},\\
		\cK_\out&= \{(x,t):x\in K_\out(t),t\geq 0\}.
	\end{align*}
	We then define the {\it inner flow} of $M(0)$,
	$$M_\ins(t)=\{x:(x,t)\in \partial\cK_\ins\}$$
	and the {\it outer flow} of $M(0)$,
	$$M_\out(t)=\{x:(x,t)\in \partial\cK_\out\}.$$
	\begin{lem}\label{lem_relate_Kin_func}
		Let $u:\R^{n+1}\x[0,\infty)\to\R$ be a level set function of $\cM$, with $u(\cdot,0)\leq 0$ on $K_\ins(0)$. Then 
		$$\R^{n+1}\backslash K_\ins(t)=\{x:u(x,t)>0\},\;\;\R^{n+1}\backslash K_\out(t)=\{x:u(x,t)<0\}.$$
	\end{lem}
	\begin{proof}
		For the first claim, we let $\Phi:\R\to\R$ by $\Phi(x)=x$ if $x>0$ and $\Phi(x)=0$ otherwise. By the relabelling lemma (\cite[Lemma 3.2]{Ilmanen92_LSF}), $v:=\Phi\circ u$ also satisfies the level set equation. Noting $v(\cdot,0)=0$ precisely on $K_\ins(0)$, which is compact, we know by the uniqueness of level set flow that $v$ is a level set function of $\cK_\ins$. Hence, 
		$$\R^{n+1}\backslash K_\ins(t)=\{x:u(x,t)>0\}.$$
		
		The second claim is similar. We let $\Psi:\R\to\R$ by $\Psi(x)=x$ if $x<0$ and $\Psi(x)=0$ otherwise. Then $v=\Psi\circ u$ satisfies the level set equation by the relabelling lemma, and $\{x:u(x,t)\geq 0\}=\{x:v(x,t)=0\}$, which is non-compact. Nevertheless, by Ilmanen \cite{Ilmanen92_LSF}, because any level sets other than $K_{\out}$ are compact, $\{x:v(x,t)=0\}$ is the biggest flow, which is unique. Then the second claim will follow.
	\end{proof}
	Finally, we denote
	$$W_\ins(t)=\R^{n+1}\backslash K_\out(t),\;\;W_\out(t)=\R^{n+1}\backslash K_\ins(t), \;\; W(t)=W_\ins(t)\cup W_\out(t).$$
	In fact, we will further define the spacetime track
	$$W_\ins[t_0,t_1]=\bigcup_{t\in [t_0,t_1]}W_\ins(t)\x \{t\},$$
	and we can similarly define $W_\out[t_0,t_1]$ and $W[t_0,t_1]$. The reason we care about these sets is that their topological changes are described by White \cite{White95_WSF_Top}, which will be crucial for us later. We remark that, when we need to specify the flow $\cM$, we will add a superscript $^\cM$ to the symbols: e.g. we will write $W^\cM_\ins(t)$ in place of $W_\ins(t)$.

	Let $(x,T)$ be a singularity of $\cM$, and $\lambda_j\to\infty$. Then any subsequential limit, in the sense of Brakke flow (see \cite[Section 7]{Ilmanen94_EllipReg}), of the rescaled flows
	$$\{\lambda_j(M(\lambda^{-2}_jt+T)-x)\}_{-\lambda^2_jT< t<0}$$
	is called a {\it tangent flow} at $(x,T)$. The tangent flow is unique if it is the shrinking cylinder or has only conical ends, by  Colding-Minicozzi \cite{CM15_Lojasiewicz} and Chodosh-Choi-Schulze \cite{ChodoshSchulze21_ConicalShrinkerUniqueness} 
	respectively.  Moreover, the convergence is in $C_\loc^\infty$ by Brakke's regularity theorem (see \cite{White05_MCFReg}).
	
	Now, following \cite{ChoiHaslhoferHershkovitsWhite22_AncientMCF}, we call $(x,T)$ an {\it inward neck singularity} of $\cM$ if, as $\lambda\to\infty$, the rescaled flows
	$$\{\lambda(K_\ins(\lambda^{-2}t+T)-x)\}_{-\lambda^2T< t<0}$$
	converge locally smoothly with multiplicity one
	to the  solid shrinking cylinder
	$$\{B^{n}(\sqrt{-2(n-1)t})\x\R\}_{t<0}$$
	up to rotation and translation.
	Similarly, we can define an {\it outward neck singularity}. If, instead, those rescaled flows converge with multiplicity one to the solid shrinking ball
	$$\{B^{n+1}(\sqrt{-2nt})\}_{t<0}$$
	up to translation,
	then we call $(x,T)$ an {\it inward spherical singularity.} We can again similarly  define an {\it outward spherical singularity}.

	\subsection{MCF through cylindrical and spherical singularities}\label{SS:MCF through cylindrical and spherical singularities}

	If every singularity of $\cM$ is a neck or a spherical singularity, then we call $\cM$ a {\it MCF through cylindrical and spherical singularities}. In this case, building on Hershkovits-White  \cite{HershkovitsWhite20_Nonfattening},    Choi-Haslhofer-Hershkovits-White showed  $M(t), M_\ins(t)$, and $M_\out (t)$ are all the same \cite[Theorem 1.19]{ChoiHaslhoferHershkovitsWhite22_AncientMCF}, i.e. fattening does not occur.

	Neck singularities are well-understood after the work of many researchers \cite{HuiskenSinestrari99_Acta, HuiskenSinestrari99_CVPDE, White00, White03, ShengWang09_SingMCF, WangXJ11_ConvexMCF, Andrews12_Noncollapsing, Brendle15_inscribed, CM15_Lojasiewicz, HaslhoferKleiner17, ADS19, ADS20, ChoiHaslhoferHershkovits18_MeanConvNeighb, ChoiHaslhoferHershkovitsWhite22_AncientMCF}, among others. In Theorem \ref{thm_canonical_nbd}, we will state the canonical neighborhood theorem of Choi-Haslhofer-Hershkovits-White \cite{ChoiHaslhoferHershkovitsWhite22_AncientMCF}. Using that, we obtain a more detailed topological description of neck singularities in Theorem \ref{thm_topological_canonical}.

	\begin{defn}\label{def_spacetime_canoncial_nbd}
		Let  $X=(x,T)$ be a regular point in a level-set flow $\cM$. Let $\lambda:=|{\bf H}(x)|$. Suppose  there exists an ancient  MCF  $\{\Sigma(t)\}$ that is, up to spacetime translation and parabolic rescaling,  one of the following:
		\begin{itemize}
			\item the shrinking sphere,
			\item the shrinking cylinder with axis $\ell$,
			\item the translating bowl with axis $\ell$,
			\item the ancient oval with axis $\ell$,
		\end{itemize} 
		such that: 
		For each  $t\in (-1/\epsilon^2,0]$ and inside $B_{1/\epsilon}(0)\subset \R^{n+1}$, 
		$$\lambda(M(\lambda^{-2}t+T)-x)\textrm{ and }\Sigma(t)$$
		are $\epsilon$-close in $C^{\lfloor{1/\epsilon}\rfloor}$.
		Then, we call $$\left(T-\frac 1{ \lambda^2\epsilon^2},T\right]\x B_{\frac 1{\lambda\epsilon}}(x)$$
		an {\it $\epsilon$-canonical neighborhood of $X$ with axis $\ell$}.
	\end{defn}
	
	We will also have a weaker definition, for  situations when we focus on a time slice:
	\begin{defn}\label{def_space_canoncial_nbd}
		Let  $x$ be a regular point in a subset $M$.  Let $\lambda:=|{\bf H}(x)|$. Suppose there exists a hypersurface $\Sigma$ that is, up to  translation and rescaling, a time slice of one of the following:
		\begin{itemize}
			\item the shrinking sphere,
			\item the shrinking cylinder with axis $\ell$,
			\item the translating bowl with axis $\ell$,
			\item the ancient oval with axis $\ell$,
		\end{itemize} 
		and such  that:  Inside $B_{1/\epsilon}(0)\subset \R^{n+1}$, 
		$\lambda(M-x)\textrm{ and }\Sigma$
		are $\epsilon$-close in $C^{\lfloor{1/\epsilon}\rfloor}$.
		Then, we call $B_{\frac 1{\lambda\epsilon}}(x)$
		an {\it $\epsilon$-canonical neighborhood of $x$ with axis $\ell$}.
	\end{defn}
	One can compare the above with the notion of $\epsilon$-canonical neighborhoods in 3-dimensional Ricci flow \cite[Lecture 2]{MorganFong10_ricci_flow_book}.
	
	\begin{thm}[Canonical neighborhood]\label{thm_canonical_nbd}
		Let  $(x,T)$ be a  neck singularity of a MCF through cylindrical and spherical singularities $\cM$, and $\ell$ be the axis of the cylindrical tangent flow at $(x,T)$. 
		Then for every $\epsilon>0$, there exists $\delta,\bar\delta>0$ such that every regular point of $\cM$ in $ B_{2\delta}(x)\x(T-\bar\delta ,T+\bar\delta )$ has an $\epsilon$-canonical neighborhood with axis $\ell$ in the sense of Definition \ref{def_spacetime_canoncial_nbd}.
	\end{thm}
	We used balls of radius $2\delta$ (instead of $\delta$): This is solely for the sake of notational convenience, so that it can be directly quoted in Theorem  \ref{thm_topological_canonical}.
	\begin{proof}
		This is from  \cite[Corollary 1.18]{ChoiHaslhoferHershkovitsWhite22_AncientMCF}. Note that all limit flows at $(x,T)$ have the same axis \cite[p.163]{ChoiHaslhoferHershkovitsWhite22_AncientMCF}.
	\end{proof}
	
	\subsection{Consequence of almost all time regularity}
	Recall that throughout this paper, a cylindrical singularity has tangent flow given by the cylinder $S^{n-1}\times\R$. By White's stratification of singular set of MCF (\cite{White97_Stratif, White03}), at almost every time, the time-slice of a MCF through cylindrical and spherical singularities is smooth.
	%, by Colding-Minicozzi \cite[Corollary 0.6]{ColdingMinicozzi16_SingularSet} 
	Based on this, in items (\ref{thm_item_transverse}) - (\ref{thm_item_each_cnn_comp_genus0}) of the following theorem, we will obtain a topologically more refined picture of neck-pinches. The shapes of the surfaces described in items (\ref{thm_item_transverse}) - (\ref{thm_item_each_cnn_comp_genus0}) are illustrated in Figure \ref{fig::item_36}.
	
	\begin{figure}[h]
		\centering
		\makebox[\textwidth][c]{\includegraphics[width=4in]{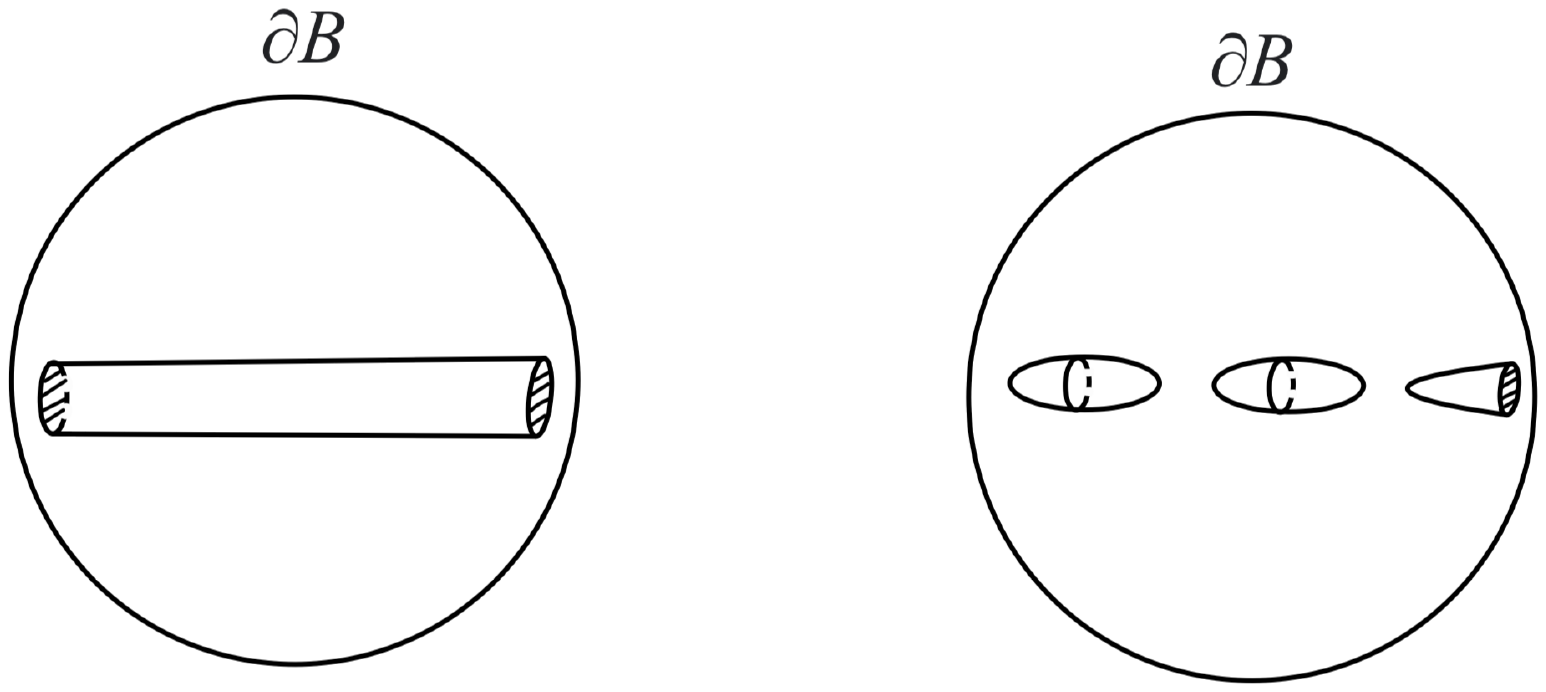}}
		\caption{}
		\label{fig::item_36}
	\end{figure}
	
	\begin{thm}\label{thm_topological_canonical} There exists a universal constant $R_0=R_0(n)$ with the following significance.
		Let  $(x,T)$ be an inward neck singularity of a MCF through cylindrical and spherical singularities $\cM$ in $\R^{n+1}$, and $\ell$ be the axis of the cylindrical tangent flow at $(x,T)$.
		For every $\delta_0>0$ and every $R>R_0$, there exists $\delta\in(0,\delta_0)$ and $\bar\delta>0$ such that:
		\begin{enumerate}
			\item \label{thm_item_is_colid_cylinder} Let $B=B_\delta(x)$. Then the set $M(T-\bar\delta)\cap B$
			\begin{itemize}
				\item is up to scaling and translation $\frac 1 {R}$-close  in $C^\infty$ to  the   cylinder ($\cong S^{n-1}\x\R$) in $B_R(0)$ with axis $\ell$ and radius $1$,
				\item and as a topological cylinder has  $K_\ins(T-\bar\delta)\cap B$ on its inside.
				\item $\bar\delta\to 0$ as $R\to\infty$.
			\end{itemize}
			\item \label{thm_item_convex} (Mean convex neighborhood) For every $T-\bar\delta <t_1<t_2<T+\bar\delta $, 
			$$K_\ins(t_2)\cap B\subset K_\ins(t_1)\backslash M(t_1).$$
		\end{enumerate}
		Moreover, there exists some countable dense set $J\subset [T-\bar\delta ,T+\bar\delta ]$  with $T-\bar \delta\in J$ such that we have for every $t\in J$:
		\begin{enumerate}\setcounter{enumi}{2}
			\item \label{thm_item_transverse}
			$M(t)$ is smooth, and  intersects $\partial B$ transversely.
			
			\item\label{thm_item_convex_loop} Each connected component of $K_\ins(t)\cap \partial B$ is a  convex $n$-ball in $\partial B$. 
			\item \label{thm_item_at_most_two_convex_disks} Denote the two connected components of $K_\ins(T-\bar\delta)\cap\partial B$ by $D_1$ and $D_2$. Then
			$M(t)\cap  D_i$ has at most one connected component for $i=1,2$. 
			\item \label{thm_item_each_cnn_comp_genus0}
			Let $K$ be a  connected component of $K_\ins(t)\cap B$. Then $K$ satisfies one of the following:
			\begin{itemize}
				\item $\partial K$ is a connected component of $M(t)\cap B$ that is a sphere.
				\item $\partial K$ consists of a connected component of $M(t)\cap B$ that is an $n$-ball and another ball on $\partial B$.
				\item $\partial K$ consists of a connected component of $M(t)\cap B$ that is a cylinder $\cong S^{n-1}\x(0,1)$ and two balls on $\partial B$.
			\end{itemize}
			
		\end{enumerate}
		And the case for outward neck singularities is analogous.
	\end{thm}
	
	\begin{proof} We will just do the case of inward neck singularity.
		\part*{To obtain   (\ref{thm_item_is_colid_cylinder}) and (\ref{thm_item_convex}).}
		Let us first arbitrarily pick some $\epsilon,R>0$,  which we will further specify later. Let $\delta,\bar\delta>0$ be obtained from applying the canonical neighborhood theorem (Theorem \ref{thm_canonical_nbd}) to $(x,T)$ and $\epsilon$. We can decrease $\bar\delta$ such that it lies in the range $ (0,\delta_0)$.
		
		By possibly further decreasing $\delta,\bar\delta$, we can guarantee (\ref{thm_item_convex}) by  the mean convex neighborhood theorem of Choi-Haslhofer-Hershkovits-White \cite[Theorem 1.17]{ChoiHaslhoferHershkovitsWhite22_AncientMCF}. In fact, further decreasing $\delta,\bar\delta$,  we can by the definition of neck singularity assume that 
		$M(T-\bar\delta)\cap B_{2\delta}(x)$
		\begin{itemize}
			\item is, up to scaling and translation, $\frac 1 {R}$-close  in $C^\infty$ to  the   cylinder ($\cong S^{n-1}\x\R$) in $B_{2R}(0)$ with axis $\ell$ and radius $1$,
			\item and as a topological cylinder has $K_\ins(T-\bar \delta)\cap B_{2\delta}(x)$ on its inside.
		\end{itemize}In particular, (\ref{thm_item_is_colid_cylinder}) is fulfilled.
		
		\part*{To define $J$ and obtain (\ref{thm_item_transverse}).} Note that using \cite[Corollary 0.6]{ColdingMinicozzi16_SingularSet}, for some set $I_1\subset [T-\bar\delta ,T+\bar\delta ]$ of full measure,  $M(t)$ is smooth for all $t\in I_1$. Then (\ref{thm_item_transverse}) just follows from a standard transversality argument. Namely, for each $t\in I_1$, via the transversality theorem, $B_r(x)$ intersects $M(t)$ transversely for a.e. $r\in (\delta/2,\delta)$. Hence, for some countable dense subset  $J\subset I_1$ and some set $I_2\subset (\delta/2,\delta)$ of full measure, for all  $(t,r)\in J\x I_2$, $B_r(x)$ intersects $M(t)$ transversely. Hence, by slightly decreasing $\delta$, (\ref{thm_item_transverse}) can be fulfilled.

		\part*{To obtain   (\ref{thm_item_convex_loop}).}
		Let us first state a lemma, which gives us the constant $R_0$ we need.
		
		\begin{lem}\label{lem_intersection_sphere}
			There exist constants $R_0>2$, and $\epsilon_0,\epsilon_1>0$, all depending only on $n$, with the following significance. 
			\begin{itemize}
				\item Consider some ball $B_{2R_0}(x)$, and fix a diameter line $\ell$. Let $\cC\subset B_{2R_0}(x)$ be the solid cylinder with radius $2$ and axis $\ell$.
				\item Let  $x'$ be a regular point of some time-slice  $M(t)$ of a level set flow in $\R^{n+1}$, and $x'$ has an $\epsilon_0$-canonical neighborhood  with axis $\ell$.      
				\item Assume $x'\in B
				_{R_0}(x)$, $M(t)\cap B_{2R_0}(x)\subset \cC$.   
				\item  Let $S$ be a smooth  $n$-disc  properly embedded in $\cC$, with $\partial S$  lying on and transversely intersecting the cylindrical part of $\partial \cC$, and $x' \in S$, such that:
				\item $S$ is $\epsilon_1$-close in $C^\infty$ to some planar $n$-disc perpendicular to $\ell$. (See Figure \ref{fig::lem_intersect}.) 
			\end{itemize}
			
			Then we have:
			\begin{itemize}
				\item  If $M(t)$ intersects $S$ transversely at $x'$, then the connected component $D$ of $K_\ins(t)\cap S$ that contains $x'$ is a convex $n$-disc in  $S$, and $M(t)\cap D=\partial D$ with the  intersection being transverse.
				\item  If $M(t)$ does not intersect $S$ transversely at $x'$, then $D$ is just the point $x'$.
			\end{itemize}
		\end{lem}
		
		\begin{figure}
			\centering
			\makebox[\textwidth][c]{\includegraphics[width=3.5in]{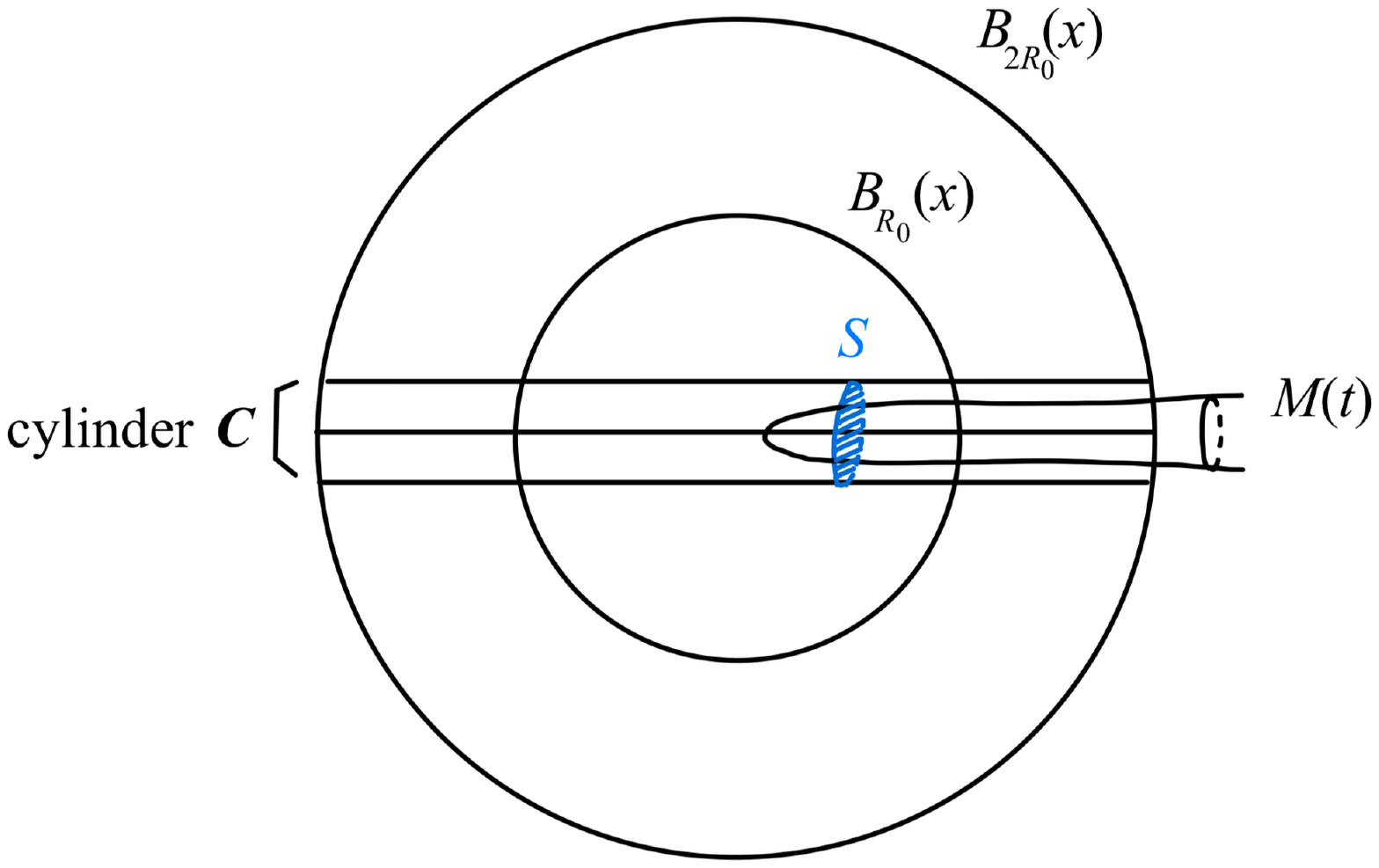}}
			\caption{}
			\label{fig::lem_intersect}
		\end{figure}
		
		\begin{proof}
			By an inspection of the geometry of the sphere, cylinder, bowl, and ancient oval, for all sufficiently large $R_0$ and small $\epsilon_0$, if   $M(t)\cap B_{2R_0}(x)\subset \cC$ then  $$M(t)\cap B_{2R_0}(x)\cap(\epsilon_0\textrm{-canonical neighborhood of }x')$$ has curvature $|A|>1/2$. Thus, if the smooth $n$-disc $S$ is sufficiently planar, the desired claim follows easily. 
		\end{proof}

		Now, we begin proving (\ref{thm_item_convex_loop}). Let us assume the $R,\epsilon$ we chose satisfy $R>R_0$ and $\epsilon<\epsilon_0$, with $R_0,\epsilon_0$ from the above lemma. By how we chose $R$ in the proof of (\ref{thm_item_is_colid_cylinder}) above, we can rescale $M(T-\bar\delta)$ by some factor $\lambda$  such that
		$$\lambda( M(T-\bar\delta)-x)\cap B_{2R}(0)$$ lies in the solid  cylinder $C\subset B_{2R}(0)$ with axis $\ell$ and radius $2$. Thus, by the mean convex neighborhood property (\ref{thm_item_convex}), for all $t\in(T-\bar\delta,T+\bar\delta)$,
		$$\lambda( M(t)-x)\cap B_{2R}(0)\subset C.$$ 
		
		Now, remember that we should focus on those $t\in J\subset (T-\bar\delta,T+\bar\delta)$.
		By  Theorem \ref{thm_canonical_nbd} and $\epsilon<\epsilon_0$, $M(t)$ has an $\epsilon_0$-canonical neighborhood with $\ell$, and so does $\lambda(M(t)-x)$ since the property is independent of scaling and translation.
		Let $S$ be a connected component of $\partial B_{R}(0)\cap C$.
		By increasing $R$, we can make  $S$ arbitrarily close to being planar. Hence, we can apply Lemma \ref{lem_intersection_sphere}. Then (\ref{thm_item_convex_loop}) follows immediately.

		\part*{To obtain (\ref{thm_item_at_most_two_convex_disks}).}
		We will just do the case for $D_1$. 
		Let
		$$T_1:=\sup\{t\in J:M(t)\cap  D_1 \textrm{ has only one connected component}\}.$$
		Note that $T_1>T-\bar\delta$ by (\ref{thm_item_is_colid_cylinder}) and $T-\bar\delta\in J$. To prove that $M(t)\cap  D_1$ has at most one connected component for each $t\in J$,  it suffices to prove that $T_1=T+\bar\delta$. Suppose the otherwise, i.e. $T_1<T+\bar\delta$ so that there exists a sequence in $J$, $t_1,t_2,...\downarrow T_1$, such that $M(t_i)\cap D_1$ contains at least two components.
		
		Now, let $$K_1=  \bigcap_{T-\bar\delta<t<T_1}K_\ins(t)\cap D_1,\;\;K_2=  K_\ins (T_1)\cap D_1,\;\; K_3=\bigcup_{i}K_\ins(t_i)\cap D_1.$$
		Note that $K_1\supset K_2\supset K_3$ by the mean convex neighborhood property (\ref{thm_item_convex}).
		\begin{prop}\label{prop_dense_in}
			$K_1$ is a convex $n$-ball in $\partial B$,   $K_1=K_2$,  and $K_3$ is dense in $K_1$.
		\end{prop}
		\begin{proof}
			By the mean convex property, $$K_1=\bigcap_{t\in J, t<T_1} K_\ins(t)\cap D_1.$$
			Then by 
			(\ref{thm_item_convex_loop}), $K_1$ is a convex $n$-ball.
			
			To prove $K_1=K_2$, it suffices to prove $K_1\subset K_2$. Note that by Lemma \ref{lem_relate_Kin_func},  for every $x\in K_1$ and $t\in (T-\bar\delta,T_1)$  we have $u(x,t)\leq 0$, where $u$ is a level set function for $\cM$. Since $u$ is continuous, $u(x,T_1)\leq 0$, implying $x\in K_2$ by Lemma \ref{lem_relate_Kin_func}.
			
			Finally, to prove $K_3$ is dense in $K_1$, it suffices to prove $K_1\backslash K_3$ has empty interior (as a subset of $\partial B$) since $K_1$ is a convex $n$-ball.
			We claim that
			$K_2\backslash K_3\subset M_\ins(T_1)$. Indeed, if $x\in K_2\backslash K_3$, then for every spacetime neighborhood $U$ of $(x,T_1)$ in $\R^{n+1}\x \R$, for each $i$, $U$ contains the point
			$$(x,t_i)\in (\R^{n+1}\x\R)\backslash \cK_\ins.$$
			Thus, $(x,T_1)\in\partial \cK_\ins$, and so $x\in M_\ins(T_1)$.
			
			As a result, 
			$$K_1\backslash K_3=K_2\backslash K_3\subset M_\ins(T_1)\cap D_1=M(T_1)\cap D_1,$$
			where the last equality is by the non-fattening of $\cM$ \cite[Theorem 1.19]{ChoiHaslhoferHershkovitsWhite22_AncientMCF}. We will prove that  $M(T_1)\cap D_1$ consists entirely of singularities (of $\cM$), and then immediately we would know $M(T_1)\cap D_1$ has empty interior using \cite[Theorem 0.1]{ColdingMinicozzi16_SingularSet}, which says that the singular set of $\cM$ is contained in finitely many compact embedded Lipschitz submanifolds each of
			dimension at most $n-1$ together with a set of dimension $n-2$.
			
			Suppose by contradiction that $M(T_1)\cap D_1$ contains some regular point $p$. So around some neighborhood of $p$ in $\R^{n+1}$, $M(T_1)$ is a smooth surface, with $K_\ins(T_1)$ on one side.
			Thus, we have $p\in\partial K_2$, with $K_2$ a convex $n$-ball. 
			Then we repeat the argument in the above proof of (\ref{thm_item_convex_loop}) to apply Lemma \ref{lem_intersection_sphere} around the point $p$, and conclude:
			\begin{itemize}
				\item $\partial K_2$ is a smooth $(n-1)$-sphere and consists entirely of regular points.
				\item The interior of $K_2$ does not intersects $M(T_1)$.
				\item  $M(T_1)$ intersects $D_1$ transversely along $\partial K_2$.
			\end{itemize} 
			So, for some short amount of time after $T_1$, $M(T_1)\cap D_1$ would still have only one connected component by pseudolocality of (locally) smooth MCF (see \cite[Theorem 1.5]{IlmanenNevesSchulze19_network}). This contradicts the definition of $T_1$.
		\end{proof}
		
		Let us continue the proof of (\ref{thm_item_at_most_two_convex_disks}). Now, for each $i$, $K_\ins(t_i)\cap D_1$  has  finitely many  connected components by transversality (\ref{thm_item_transverse}). Let $E_i$ be the one with the maximal diameter (measured inside $\partial B$), denoted $d_i$. Then by the canonical neighborhood property Theorem \ref{thm_canonical_nbd}, assuming $\epsilon$ small, for some geodesic ball $\tilde E_i\subset \partial B$ of diameter $3d_i$, $\tilde E_i\cap K_\ins(t_i)=E_i$.
		
		Now, note $d_i$ is increasing in $i$ by the mean convex neighborhood property (\ref{thm_item_convex}). Let $d=\lim_i d_i$. There are two cases: (a) $d\geq \diam(K_1)/2$, and (b) $d<\diam(K_1)/2$. For case (a), by the definition of  $t_i$, we know for sufficiently large $i$, the neighborhood $\tilde E_i$ would then need to contain a connected component of $K_\ins(t_i)\cap D_1$ other than $E_i$, contradicting the definition of $\tilde E_i$. So case (a) is impossible. Case (b) is also impossible since it, together with the existence of $\tilde E_i$, violates  Proposition \ref{prop_dense_in} which says $K_3$ is dense in $K_1$.
		This finishes the proof of (\ref{thm_item_at_most_two_convex_disks}).
		
		\part*{To obtain (\ref{thm_item_each_cnn_comp_genus0}).}
		Choose a connected component $K$ of $K_\ins(t)\cap B_{\delta}(x)$. Let us foliate  $B_{2\delta}(x)$ with planar $n$-discs that are perpendicular to the axis $\ell$. Then as in the proof of (\ref{thm_item_convex_loop}), we apply Lemma \ref{lem_intersection_sphere} to characterize the intersection of $K$ with every such planar $n$-discs. Namely, every such set of intersections consists of convex $n$-discs and isolated points.
		Viewing these sets of intersection as level sets of some function defined on $K$, Morse theory then immediately implies (\ref{thm_item_each_cnn_comp_genus0}).

		This finishes the proof of Theorem \ref{thm_topological_canonical}.
	\end{proof}
	
	Finally, we discuss some convergence theorems of MCF through cylindrical and spherical singularities. 
	
	\begin{prop}\label{lem_ae_time_smooth_conv}
		Let $\cM^i=\{M^i(t)\}_{t\geq 0}$, with $i=1,2,...$, and  $\cM=\{M(t)\}_{t\geq 0}$ be MCF through neck and spherical singularities in $\R^{n+1}$. Assume that each $M^i(0)$ and $M(0)$ are smooth, closed hypersurfaces, with $M^i(0)\to M(0)$ in $C^\infty$. Then 
		\begin{enumerate}
			\item For a.e. $t$, $M^{i}(t)\to M(t)$ in $C^\infty$.
			\item The spacetime tracks $\cM^i\to\cM$ in the Hausdorff sense.
		\end{enumerate}
	\end{prop}
	
	\begin{proof}
		By Ilmanen's elliptic regularization (see \cite{Ilmanen94_EllipReg, White09_CurrentsVarifolds}), for any closed smooth hypersurface $M^i(0)$, there exists a unit regular cyclic Brakke flow $\{\mu_t^i\}_{t\geq 0}$ such that $\mu_0^i=M^i(0)\lfloor\cH^n$, where $\cH^n$ is the $n$-dimensional Hausdorff measure. By the mean convex neighborhood theorem \cite{ChoiHaslhoferHershkovits18_MeanConvNeighb} and the nonfattening of level set flow with singularities that have mean convex neighborhood \cite{HershkovitsWhite20_Nonfattening}, $\{\mu^i_t\}_{t\geq 0}$ is supported on $\cM^i$. Then the compactness of Brakke flows (\cite{Ilmanen94_EllipReg, White09_CurrentsVarifolds}) implies that $\{\mu^i_t\}_{t\geq 0}$ subsequentially converges to a limit unit regular cyclic Brakke flow $\{\mu^\infty_t\}_{t\geq 0}$.
		
		Because $M^i(0)\to M(0)$ smoothly, $\mu^\infty_0=\mu_0$, and by the uniqueness of unit regular cyclic Brakke flow, $\mu^\infty_t=\mu_t$ a.e. for all $t\geq 0$. In particular, the regular part of $\mu^\infty_t$ equals the regular part of $\mu_t$. Then by Brakke's regularity theorem and a.e. time regularity of $\cM^i$ with neck and spherical singularities, we have for a.e. $t$, $M^i(t)\to M(t)$.
		
		The compactness of weak set flow shows that $\cM^i$ subsequentially converges to a limit weak set flow $\cM^\infty$ in Hausdorff distance. Because $\{\mu_t\}_{t\geq 0}$ is supported on $\cM^\infty$, we have $\cM\subset\cM^\infty$. Meanwhile, $\cM$ is the biggest flow, therefore $\cM^\infty\subset\cM$. Thus, $\cM^\infty=\cM$. This also shows the uniqueness of the limit. Therefore, $\cM^i$ converges to $\cM$ in Hausdorff distance.
		
	\end{proof}
	
	\section{Homology descent, homology termination, and homology breakage}\label{sect_general_results}
	
	In this section, we consider general level set flows $\cM=\{M(t)\}_{t\geq 0}$ in $\R^{n+1}$, where $M(0)$ is not necessarily a closed hypersurface. We will introduce three new concepts. For a heuristic explanation of them, see \S \ref{sect_topo_mcf}.

	Let $H_k(\cdot)$ denotes the $k$-th homology group in $\Z$-coefficients.
	
	\begin{defn}[Homology descent]\label{defn_order}
		We define a  relation $\succ$ on the {\it disjoint} union $$\bigsqcup_{t\geq 0}H_{n-1}(W(t) )$$ as follows. Given two times $ T_0\leq T_1$, and two  homology  classes $c_0\in H_{n-1}(W(T_0) )$ and $c_1\in H_{n-1}(W(T_1) )$, we say that {\it $c_1$ descends from $c_0$}, and denote $$c_0\succ c_1,$$ if every representative $\gamma_0\in c_0$ and $\gamma_1\in c_1$ together bound some $n$-chain $\Gamma\subset W[T_0,T_1]$, i.e. $\gamma_0-\gamma_1=\partial\Gamma.$ (See Figure \ref{fig::descent}.)
	\end{defn}
	Clearly, in the above definition, we can interchangeably replace ``every representative" with ``some representative". Note that we are using singular homology, which means that  $\gamma_0, \gamma_1,$ and $\Gamma$ are just singular chains.
	
	\begin{rmk}\label{rmk}
		The relation $\succ$ is a partial order. Indeed,
		let $c_i\in H_{n-1}(W(T_i) )$ for $i=0,1,2$.
		Clearly $c_0\succ c_0$.
		If $c_0\succ c_1$ and $c_1\succ c_0$, then $T_0=T_1$, implying $c_0=c_1$.
		Moreover, if $c_0\succ c_1$ and $c_1\succ c_2$, then $T_0\leq T_2$ and it readily follows from definition that $c_0\succ c_2$.
	\end{rmk}
	
	This relation has certain favorable properties.
	
	\begin{prop}\label{prop_homology_unique}
		Let $c_0\in H_{n-1}(W(T_0) )$ and $T_0\leq T_1$. Then there exists at most one  $c_1\in H_{n-1}(W(T_1) )$ such that $c_0\succ c_1$.
	\end{prop}
	\begin{proof}
		Suppose $c_1,c_2\in H_{n-1}(W(T_1) )$ satisfy $c_0\succ c_1$ and $c_0\succ c_2$. Our aim is to show $c_1=c_2$. Choose $\gamma_i\in c_i$ for $i=0,1,2$. Then by definition, $\gamma_0-\gamma_1=\partial A$ for some $A\subset W[T_0,T_1]$, and similarly $\gamma_0-\gamma_2=\partial B$ for some $B\subset W[T_0,T_1]$. Thus, $\gamma_1$ and $\gamma_2$ bound $A-B\subset W[T_0,T_1]$.
		Since the map $$H_{n-1}(W(T_1))\to H_{n-1}(W[T_0,T_1])$$ induced by the inclusion
		$W(T_1)\to W[T_0,T_1]$ is injective by White \cite[Theorem 1 (iii)]{White95_WSF_Top}, we deduce that $\gamma_1$ and $\gamma_2$ are homologous within $W(T_1)$. Consequently, $c_1=c_2$.
	\end{proof}
	
	\begin{rmk}
		Note that in the above it is possible that there does not exist any $c_1\in H_{n-1}(W(T_1))$ for which $c_0\succ c_1$. As illustrated in Figure \ref{fig::a_break}, after time $T$, no homology class $c_1$ satisfies $a_0\succ c_1$.
	\end{rmk}

	\begin{figure}[h]
		\centering
		{\includegraphics[width=5in]{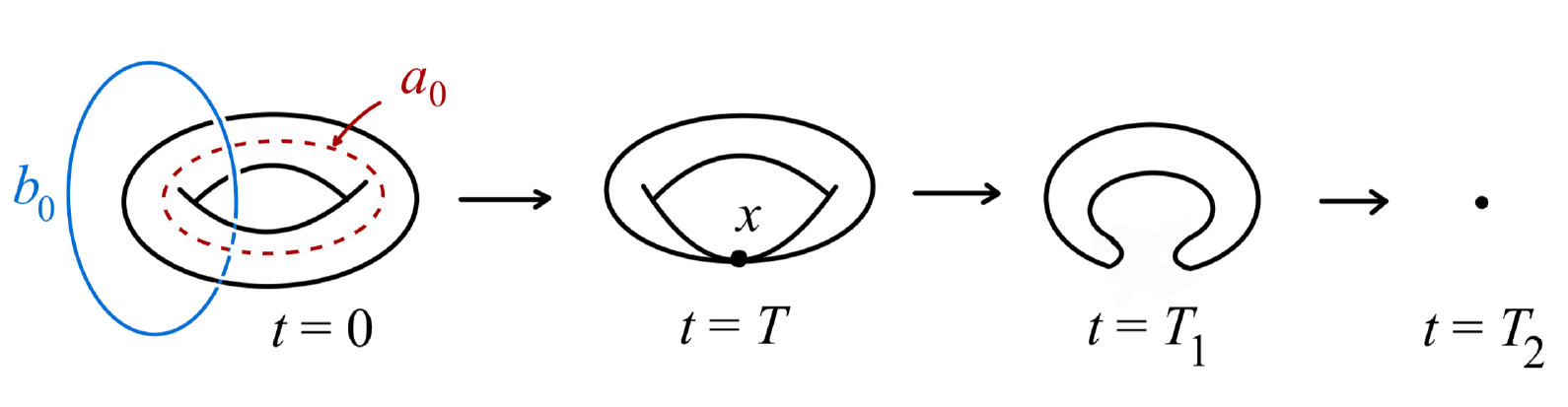}}
		\caption{}
		\label{fig::a_break}
	\end{figure}

	\begin{rmk} \label{exmp_more_into_one}
		On the other hand, there may be multiple homology classes $c_0\in H_1(W(T_0))$ satisfying the relation $c_0\succ c_1$. As an example, consider the flow shown in Figure \ref{fig::a_break}, where both $b_0\in H_1(W_\out(0))$ and the trivial element of $H_1(W_\out(0))$ descend to the trivial element of $H_1(W_\out(T_1))$.
	\end{rmk}

	In fact, precisely because of Proposition \ref{prop_homology_unique} and Remark \ref{exmp_more_into_one},  we chose the symbol $\succ$ (instead of $\prec$)  to pictographically reflect that more than one homology class may descend into one, but not the other way around.

	\begin{prop}\label{prop_at_least_one}
		We focus on the case $n=2$. Let $c_1\in H_1(W(T_1) )$ and $T_0\leq T_1$. Then there exists at least one   $c_0\in H_1(W(T_0) )$ such that $c_0\succ c_1$.
	\end{prop}
	\begin{proof}
		Choose some $\gamma\in c_1$. By White \cite[Theorem 1 (ii)]{White95_WSF_Top}, $\gamma$ can be homotoped through $W[T_0,T_1]$ to some loop $\gamma'$ in $W(T_0)$. So $c_0:=[\gamma']\succ c_1$.
	\end{proof}

	The following proposition says that a homology class cannot disappear and then reappear later.
	\begin{prop}\label{prop_unique_clss_inbetween}
		Let $T_0<T_1$, $c_0\in H_{n-1}(W(T_0))$, and  $c_1\in H_{n-1}(W(T_1))$ with $c_0\succ c_1$. Then for every $t\in [T_0,T_1]$ there exists a unique $c\in H_1(W(t))$ such that $c_0\succ c\succ c_1$.
	\end{prop}
	\begin{proof}We only need to prove existence, as then uniqueness would follow from Proposition \ref{prop_homology_unique}.
		
		Under our assumption, we have $\gamma_0\in c_0$ and $\gamma_1\in c_1$ such that they together bound some $n$-chain $C$ in $W[T_0,T_1]$. 
		Since $W[T_0, T_1]$ is an open subset of Euclidean space, we can choose a representative of the $n$-chain $C$ as a polyhedron chain. By tilting the faces appropriately, we can ensure that they do not lie entirely within any specific slice $\mathbb{R}^{n+1} \times \{t\}$. This enables us to find $\beta_t = \{x : (x, t) \in C\}$ as an $(n-1)$-chain without a boundary for each $t \in [T_0, T_1]$. Consequently, we have $[\beta_t] \in H_{n-1}(W(t))$, and $c_0 \succ [\beta_t] \succ c_1$.
	\end{proof}
	
	Based on Proposition \ref{prop_unique_clss_inbetween}, the following definition is well-defined.
	\begin{defn}[Homology termination]\label{defn_termination}
		Let $c_0\in H_{n-1}(W(T_0))$. 
		\begin{itemize}
			\item If
			$$\mathfrak{t}(c_0):=\sup\{t\geq T_0:c_0\succ c \textrm{ for some }c\in H_{n-1}(W(t))\}$$
			is finite, then we say that $c_0$ {\it terminates at time} $\mathfrak{t}(c_0)$, otherwise we say $c_0$ {\it never terminates}.
			\item And for each $t\geq T_0$, the unique $c\in H_{n-1}(W(t))$ such that $c_0\succ c$, if exists, is denoted $c_0(t)$.
		\end{itemize}
		If needed, we use $\ft^\cM$ in place of $\ft$ to specify the flow.
	\end{defn}
	Note that since $W$ is open, if $c_0$ terminates at time $\mathfrak{t}(c_0)$ then there is no $c\in H_{n-1}(W(\mathfrak{t}(c_0)))$ such that $c_0\succ c$. So $c_0(\mathfrak{t}(c_0))$ is not well-defined, and every $c_0\in H_{n-1}(W(T_0))$ does not terminate at time $T_0$. Therefore, one can interpret the time interval $[T_0,\mathfrak{t}(c_0))$ as the ``maximal interval of existence'' for $c_0$.
	
	\begin{rmk}[Trivial homology classes]
		Let us also elaborate on trivial homology classes. At each time $t$, $H_{n-1}(W(t))$ has a {\it unique} trivial homology class $0_t$. This is true even for situations like Figure \ref{fig::a_break} when the surfaces have inside and outside regions: The trivial elements of $H_1(W_\ins(t))$ and $H_1(W_\out(t))$ are viewed as the same.
		
		However, $0_t$  is considered distinct for different $t$, because we used disjoint union in Definition \ref{defn_order}. Nonetheless, for any $t_1<t_2$, it is vacuously true that $0_{t_1}\succ 0_{t_2}$. Thus, we can denote each $0_t$ as $0(t)$, following the notation in Definition \ref{defn_termination}. In addition, clearly, {\it the trivial homology class never terminates}.
	\end{rmk}

	\begin{exmp}
		Let us revisit Figure \ref{fig::a_break}. It is clear that $a_0$ terminates at time $T$, whereas $b_0$ does not. In fact, $b_0$ will never terminate: $b_0(t)$ would just become trivial for each $t>T$.
	\end{exmp}
	
	\begin{exmp}
		Let us now instead consider the flow in Figure \ref{fig::b_break}. At time $T$, $b_0$ terminates while  $a_0$ does not. In fact, $a_0(t)$ becomes trivial after time $T$, and thus it will never terminate.
	\end{exmp}
	\begin{figure}[h]
		\centering
		{\includegraphics[width=5in]{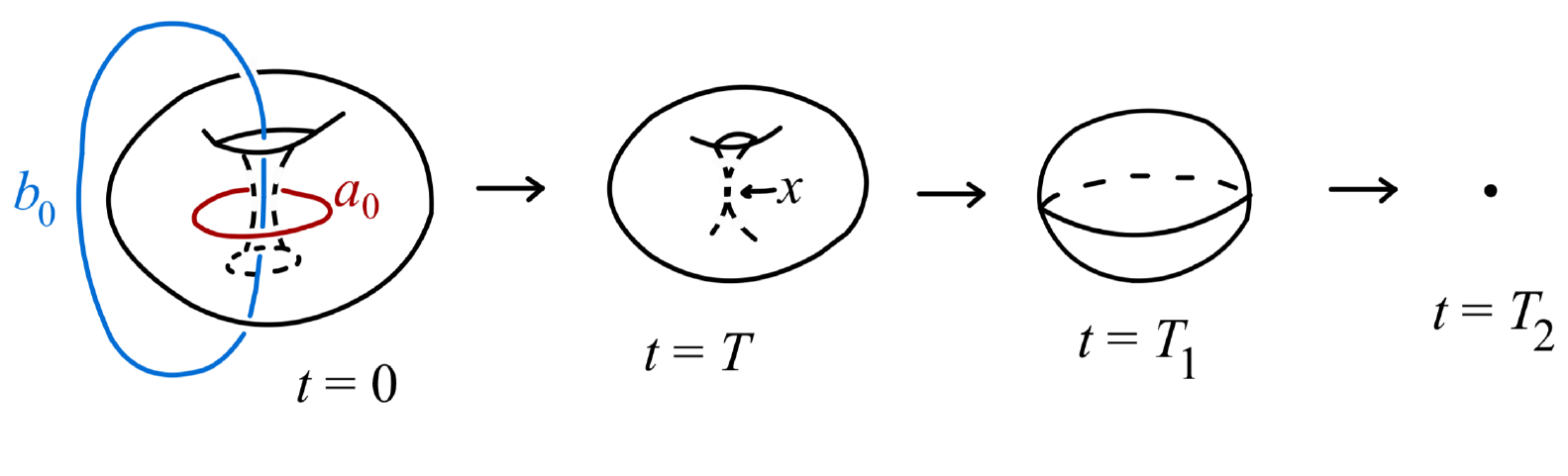}}
		\caption{}
		\label{fig::b_break}
	\end{figure}

	Now, we introduce another concept. In Figure \ref{fig::a_break}, $a_0$ terminates at time $T$ because, intuitively, it ``breaks" at the cylindrical singularity $x$. Similarly, in Figure \ref{fig::b_break}, $b_0$ terminates at time $T$ because it ``breaks" at the outward cylindrical singularity. The following definition provides a precise characterization of this breakage phenomenon.
	
	\begin{defn}[Homology breakage]\label{defn_breakage}
		Let $c_0\in H_{n-1}(W(T_0) )$,  $T_1> T_0$, and $K\subset M(T_1)$ be a compact set. Suppose  the following holds:
		\begin{itemize}
			\item For each $T_0\leq t<
			T_1$, there exists $c_0(t)\in H_{n-1}(W(t) )$ such that $c_0\succ c_0(t)$.
			\item For every neighborhood $U\subset\R^{n+1}$ of $K$, for each $t<T_1$ sufficiently  close  to $T_1$, every element of $c_0(t)$ intersects $U$. (Recall Figure \ref{fig::breakage}.)
		\end{itemize}
		Then we say that $c_0$ {\it breaks in}  $(K,T_1)$. We will often concern the case when $K$ is just a point $x\in M(T)$, for which we say that $c_0$ {\it breaks at} $(x,T_1)$.
	\end{defn}

	One might wonder why Definition \ref{defn_breakage} does not require $c_0$ to terminate at time $T_1$. This is because it is not necessary:
	\begin{prop}\label{prop_homology_class_die}
		If a homology class $c_0\in H_{n-1}(W(T_0) )$ breaks  in  some  $(K,T_1)$, then $c_0$ terminates at time $T_1$.
	\end{prop}
	
	\begin{proof}
		Suppose the otherwise: There exists $T_2>T_1$ and $c_2\in H_{n-1}(W(T_2) )$ such that $c_0\succ c_2$. Then there exists $\gamma_0\in c_0$ and $\gamma_2\in c_2$ that together in $W[T_0,T_2]$ bound some $n$-chain $C$. Without loss of generality we can assume that $\beta_t:=\{x:(x,t)\in C\}$  is an  $(n-1)$-chain without boundary for each $t\in [T_0,T_2]$, as in the proof of Proposition \ref{prop_unique_clss_inbetween}. Then $c_0(t)=[\beta_t]\in H_{n-1}(W[t])$ satisfies $c_0\succ c_0(t)$.
		
		By assumption $c_0$ breaks in some  $(K,T_1)$ with $K\subset M(T_1)$. Therefore, $K\cap C=\emptyset$. Since $K$ is compact and $C$ is closed, there exists a neighborhood of $K$ in $\R^{n+1}\x\R$ of the form $B_r(K)\x[T_1-\delta,T_1+\delta]$  that does not intersect $C$. Consequently, for all $t\in [T_1-\delta,T_1+\delta]$, $\beta_t$ avoids $ B_r(K)$.
		This contradicts the assumption that $c_0$ breaks at $(K,T_1)$. 
	\end{proof}
	
	Note that, vacuously, the trivial homology class does not break in any $(K,T)$.
	Moreover, if a homology class breaks in $(K_1,T)$ and $K_1\subset K_2\subset M(T)$, then it also breaks in $(K_2,T)$.
	
	One might wonder whether the converse of the above proposition is true. Actually, in the case of 2-dimensional MCF through cylindrical and spherical singularities, if a homology class terminates at some time $T$, then it actually breaks at some cylindrical singularity $(x,T)$. This is the statement of Theorem \ref{thm_exist_break_point}, which is one of the main results in \S \ref{sect_different_termination_time}. However, we are unsure whether the converse is true in general.

	\begin{prop}\label{prop_no_reg}
		No homology class breaks at a regular point.
	\end{prop}
	\begin{proof}
		Suppose $(x,T)$ is a regular point. Then there exists a small ball $B$ around $x$ such that for all $t$ close to $T$,  $M_t\cap B$ is a smooth $n$-disk. It is clear that every $n$-chain can be homotoped to avoid $B$. Therefore, no homology class breaks at $(x,T)$.
	\end{proof}
	
	\begin{prop}\label{prop_no_sphere}
		No homology class breaks at a spherical singularity.
	\end{prop}
	\begin{proof}
		Suppose otherwise. Without loss of generality, suppose some $c_0\in H_{n-1}(W(T_0))$ breaks at some spherical singularity $(x,T)$. Then there exists a small ball $B$ around $x$ such that for all $t<T$ close to $T$,  $M(t)\cap B$ is a smooth sphere. For each such $t$, let $\gamma$ be a representative of $c_0(t)$. By removing the components of $\gamma$ inside the sphere $M(t)\cap B$, we can assume that $\gamma$ lies outside the sphere. Thus clearly $\gamma$ can be homotoped within $W(t)$ to avoid $B$. This again contradicts the assumption that $c_0$ breaks at $(x,T)$.
	\end{proof}
	
	Lastly, we conclude this section with the following proposition, which provides us with a scenario where we know the inside homology classes must terminate. Namely, {\it if we take a compact shrinker and push it inward, then all non-trivial inside homology classes will terminate, while the outward ones will not.} This proposition will be crucial for us when we use Theorem \ref{thm:main1} to prove other main theorems.
	
	\begin{prop}\label{prop_perturb_shrinker_which_terminate_first}
		The setting is as follows.
		\begin{itemize} 
			\item  Let $\Sigma$ be a smooth, embedded,  compact  shrinker in $\R^3$.
			\item Let $S^0(-1)$ be a surface, lying strictly inside $\Sigma$, given by deforming $\Sigma$ within the inside region of $\Sigma$.
			\item Let $S^1(-1)$ be a surface, lying strictly outside $\Sigma$,  given by deforming $\Sigma$ within the outside region of $\Sigma$. 
			\item Note that the first homology groups of 
			$$\R^3\backslash \Sigma,\;\; \R^3\backslash S^0(-1),\textrm{ and }\R^3\backslash S^1(-1)$$
			can be canonically identified.
			\item Let  
			$$\cS=\{\sqrt{-t}\Sigma\}_{-1\leq t\leq 0}, \;\; \cS^0=\{S^0(t)\}_{t\geq -1}, \textrm{ and }\cS^1=\{S^1(t)\}_{t\geq -1} $$
			be the associated level set flows.
		\end{itemize}
		Then there exist times $T,\tilde T\in (-1,0)$ such that
		\begin{enumerate}
			\item  For each non-trivial element $a_0\in H_1(W_\ins^{\cS^0}(-1))$, $\ft(a_0)\leq\tilde T$. 
			\item For each  element $b_0\in H_1(W_\out^{\cS^0}(-1))$, $b_0(\tilde T)$ exists and is trivial. 
			\item For each  element $a_1\in H_1(W_\ins^{\cS^1}(-1))$, $a_1(T)$ exists and is trivial. 
			\item  For each non-trivial element $b_1\in H_1(W_\out^{\cS^1}(-1))$, $\ft(b_1)\leq T$.
		\end{enumerate}
	\end{prop}
	\begin{proof}
		
		For the first claim, 
		note that:
		\begin{itemize}
			\item  $S^0(-1)$ is inside $\Sigma$.
			\item $\dist(\sqrt{-t}\Sigma,S^0(t))$ is non-decreasing in $t$ by \cite[Theorem 7.3]{EvansSpruck91}.
			\item $\Sigma$ shrinks self-similarly under the flow.
		\end{itemize}
		Thus, we can deduce the existence of $\tilde T<0$ such that for every $t\geq \tilde T$, $S^0(t)$ is empty.
		Consequently, for any non-trivial element  $a_0\in H_1(W^{\cS^0}_\ins(-1))$, either $\ft(a_0)\leq \tilde T$, or $a_0(\tilde T)$ still exists but is trivial. Suppose by contradiction that the latter holds. Then we can pick some $\alpha_0\in a_0$ such that $\alpha_0=\partial A$ for some $$A\subset W^{\cS^0}_\ins([-1,\tilde T])\subset W^{\cS}_\ins([-1,\tilde T]).$$ By rescaling each time slice of $A$, we can ensure that $\alpha_0$ bounds some
		$$\tilde A\subset (\textrm{interior region of }\Sigma)\x[-1,\tilde T].$$
		Projecting $\tilde A$ into the interior region of $\Sigma$, we have that $\alpha_0$ is homologically trivial, which contradicts the definition of $\alpha_0$. This concludes the proof of the first claim.
		
		For the second claim, since $\Sigma$ shrinks self-similarly under the flow, we know that $b_0$ has not terminated by the time $\tilde T (< 0)$ for the flow ${\sqrt{-t}\Sigma}$. Then by the fact that $S^0(t)$ lies inside $\sqrt{-t}\Sigma$ for each $t\in[-1,\tilde T]$, which is a result of the avoidance principle, we can deduce that $b_0(\tilde T)$ still exists for the flow $\cS^0$. However, as $S^0(\tilde T)$ is empty, it follows that $b_0(\tilde T)$ must be trivial.
		
		Let us define $$\epsilon=\dist(\Sigma,S^1(-1))\}.$$
		Pick a loop $\beta_1\in b_1$. Define $B_\epsilon(\sqrt{-t}\Sigma)$ as the $\epsilon$-neighborhood of $\sqrt{-t}\Sigma$, and denote
		\begin{align*}
			Y(t)&:=\R^3\backslash B_\epsilon(\sqrt{-t}\Sigma)\\
			Y[t_1,t_2]&:=\bigcup_{t\in[t_1,t_2]}(\R^3\backslash B_\epsilon(\sqrt{-t}\Sigma))\x\{t\}.
		\end{align*}
		We prove the fourth claim before the third. In order to prove the fourth claim, it suffices to show for some $-1<T<0$, there exists no $2$-chain $C\subset W^{\cS^1}_\out[0,T]$ such that $\partial C=\beta_1- \beta_2$, where $\beta_2$ is a closed $1$-chain outside $S^1(T)$. Since $S^1(-1)$ lies outside $\Sigma$, by the avoidance principle  it suffices to prove that:
		\begin{lem}
			For some $-1<T<0$, there does not exists a $2$-chain $C\subset Y[-1,T]$ such that $\partial C=\beta_1-\beta_2$ for some closed $1$-chain $\beta_2\subset Y(T)$.
		\end{lem}
		\begin{proof}
			Choose a value of $T$ that is sufficiently close to $0$ such that $\diam(\sqrt{-T}\Sigma)<\epsilon$. With this choice, the set $B_\epsilon(\sqrt{-T}\Sigma)$ is star-shaped with respect to any point on $\sqrt{-T}\Sigma$. Thus, the boundary $\partial B_\epsilon(\sqrt{-T}\Sigma)$ has genus $0$.
			
			Suppose by contradiction that there exists a $2$-chain $C\subset Y[-1,T]$ such that $\partial C=\beta_1-\beta_2$ for some closed $1$-chain $\beta_2\subset Y(T)$. By rescaling $C$ at each time slice $t$, we can construct another $2$-chain $\tilde C$ outside $\Sigma$ such that $\partial \tilde C=\beta_1-\sqrt{-T}\beta_2$.
			
			Since $\beta_1$, which lies outside $\Sigma$, is homologically non-trivial, we can pick a non-trivial loop $\alpha$ inside $\Sigma$ such that $[\beta_1]\in H_1(\R^3\backslash \alpha)$
			is non-trivial. Then by the existence of $\tilde C$, we have $[\beta_2]\ne 0$ in $H_1(\R^3\backslash \alpha)$ too.        However, this is impossible because $\sqrt{-T}\beta_2$ lies outside $B_{\epsilon/\sqrt{-T}}(\Sigma)$ while $\alpha$ lies inside, and $\partial B_{\epsilon/{\sqrt{-T}}}(\Sigma)$ has genus $0$ by the first paragraph of this proof.
		\end{proof}
		
		This finishes proving the fourth claim of Proposition \ref{prop_perturb_shrinker_which_terminate_first}. Finally, for the third claim, since $a_1(T)$ exists for the flow $\{\sqrt{-t}\Sigma\}_{t\leq 0}$, it follows from the avoidance principle that $a_1(T)$  exists for $\cS^1$. Moreover, since the inside of $S^1(T)$ contains $B_\epsilon(\sqrt{-T}\Sigma)$, which is star-shaped, we know $a_1(T)=0 $ in $H_1(W^{\cS^1}(T))$. 
	\end{proof}

	\section{Homology breakage of MCF through cylindrical and spherical singularities}\label{sect_different_termination_time}
	
	\subsection{MCF through cylindrical and spherical singularities}\label{sect_breakage_MCF_through_sing}
	In this section, we focus on 2-dimensional MCF $\cM=\{M(t)\}_{t\geq 0}$ through cylindrical and spherical singularities in $\R^3$, where the initial condition $M(0)$ is a smooth, closed surface. 
	
	\begin{prop}\label{nobreak1}
		For any $T_0\geq 0$, no element of $ H_1(W_\out(T_0))$ can break at an inward neck singularity, and no element of $ H_1(W_\ins(T_0))$ can break at an outward neck singularity.
	\end{prop}
	\begin{proof}
		Let us just prove the first claim.
		Suppose by contradiction some $c_0\in H_1(W_\out(T_0))$  breaks at an inward neck singularity $(x,T)$, with $T>T_0$. Applying Theorem \ref{thm_topological_canonical} to $(x,T)$ with $\delta_0=1$ and any $R>R_0$, we obtain constants $\delta,\bar\delta>0$ and a dense subset $J\subset [T-\bar\delta,T+\bar\delta]$ satisfying the properties in Theorem \ref{thm_topological_canonical}. Let $B=B_\delta(x)$.

		Pick a time $t\in J\cap [T-\bar\delta,T)$. Since $c_0$ breaks at $T$, $c_0(t)$ still exists. Pick a
		loop $\gamma\in c_0(t)$. By Theorem \ref{thm_topological_canonical} (\ref{thm_item_each_cnn_comp_genus0}) (and recall  Figure \ref{fig::item_36}), we can homotope $\gamma$ within $W_\out(t)$ to avoid $B$. This can be done for all $t$ in $J\cap [T-\bar\delta,T)$, which is dense in $[T-\bar\delta,T)$. So we obtain a contradiction to the fact that $c_0$ breaks at $(x,T)$.  
	\end{proof}

	In the following proposition, we provide a more detailed description of the shape around a neck pinch at which homology class \emph{breaks}. Namely, in this case, prior to the singular time, only the last bullet point of Theorem \ref{thm_topological_canonical} (\ref{thm_item_each_cnn_comp_genus0}) can occur, i.e. $M(t)\cap B$ is a cylinder.
	\begin{prop}\label{prop_shape_of_breaking_neck}There exists a universal constant $R_0>0$ with the following significance.
		Suppose $c_0\in H_1(W_\ins(T_0))$ breaks at some inward neck singularity $(x,T)$. Let $\delta_0>0$. Then for each $R>R_0$, there exist constants
		$\delta\in(0,\delta_0)$, $\bar\delta>0$, and a dense subset $J\subset (T-\bar\delta,T+\bar\delta)$ with $T-\bar\delta\in J$, such that:
		\begin{enumerate}
			\item The first five items of Theorem \ref{thm_topological_canonical} hold.
			\item\label{prop_item_is_cylinder} For each $t\in J\cap [T-\bar\delta,T)$,
			$K_\ins(t)\cap B_\delta(x)$ is a solid cylinder such that its boundary consists of a connected component of $M(t)\cap B_\delta(x)$ that is a cylinder and two disks $D_1,D_2$ on $\partial B_\delta(x)$.
			\item\label{prop_item_non_zero_intersection_number} Moreover, for such $t$, every element $\gamma\in c_0(t)$ has a non-zero intersection  number (in $\Z$-coefficients) with each $D_i$.
		\end{enumerate}

		The outward case is analogous.
	\end{prop}
	\begin{proof}
		We will just prove the inward case. Let us apply Theorem \ref{thm_topological_canonical} to $(x,T)$ to obtain the constants $\delta,\bar\delta$ and the subset $J\subset[T-\bar\delta,T+\bar\delta]$.  Let $B=B_\delta(x)$. In addition the first five items of Theorem \ref{thm_topological_canonical} will hold.
		
		We need to show that for each $t\in J\cap(T_0,T)$ sufficiently close to $T$, $K_\ins(t)\cap B_\delta(x)$ satisfies the description in (\ref{prop_item_is_cylinder}): After that  we could just shrink $\bar\delta$ and the set $J$ to guarantee (\ref{prop_item_is_cylinder}). Suppose by contradiction that there exists a sequence in $J$, $t_1,t_2,...\uparrow T$, such that $K_\ins(t_i)\cap B_\delta(x)$ violates the description in (\ref{prop_item_is_cylinder}). Fix one $t_i$. Note that Theorem \ref{thm_topological_canonical} (\ref{thm_item_at_most_two_convex_disks}) and (\ref{thm_item_each_cnn_comp_genus0}) together imply that $K_\ins(t_i)\cap B$ can have {\it at most one} cylindrical component.
		Thus, in our case, $K_\ins(t_i)\cap B$ actually has no cylindrical component.
		As a result, any connected component
		$K$ of $K_\ins(t_i)\cap B$ satisfies either one of the following by Theorem \ref{thm_topological_canonical}   (\ref{thm_item_each_cnn_comp_genus0}): 
		\begin{itemize}
			\item $\partial K$ is a connected component of $M(t)\cap B$ that is a sphere.
			\item $\partial K$ consists of a connected component of $M(t)\cap B$ that is an disc and another disc on $\partial B$.
		\end{itemize}
		In either situation, any element of $c_0(t_i)$ can be perturbed to avoid $B$. Applying this argument to each $t_i$, we obtain a contradiction to the fact that $c_0$ breaks at $(x,T)$. 
		
		Finally, to prove (\ref{prop_item_non_zero_intersection_number}), it suffices to show that for each $t\in J\cap(T_0,T)$ sufficiently close to $T$, $c_0(t)$ satisfies the description of (\ref{prop_item_non_zero_intersection_number}): Then we could just  shrink $J$, and we would be done. Suppose otherwise, so that   there exists a sequence in $J$, $t_1,t_2,...\uparrow T$, such that $c_0(t_i)$ violates the description of 
		(\ref{prop_item_non_zero_intersection_number}). Then for each $t_i$, we can find a loop $\gamma\in c_0(t_i)$ with intersection number zero with some connected component of $K_\ins(t_i)\cap \partial B$. In fact, since $K_\ins(t_i)\cap B$ is a cylinder by (\ref{prop_item_is_cylinder}), 
		$\gamma$ has intersection number zero with {\it both}  connected components $D_1,D_2$ of $K_\ins(t_i)\cap \partial B$ (which are discs). To contradict the fact that  $c_0$ breaks at $(x,T)$, it suffices to find another element of $c_0(t_i)$
		that avoids $B$.
		
		Indeed, this can be proved as follows. We can assume $\gamma$ intersects $\partial B$ transversely. Since $\gamma$ has intersection number zero with $D_1$,
		we can pair up each positive intersection point of $\gamma\cap D_1$ with a negative one. 
		Now fix a pair, and draw a line segment $L$ on $D_1$ to connect the pair of points. Adding $L$ and $-L$ to $\gamma$, and slightly pushing the resulting curve away from $D_1$ around $L,-L$, we can obtain another representative of $c_0(t_i)$ that avoids this pair of intersection points.
		And we do this for each pair. Then at the end, we get a curve belonging to $c_0(t_i)$ that avoids $D_1$ completely. Then, we repeat this process with $D_2$, to get a curve that avoids $D_2$ too. Lastly, we discard all connected components of the curve that are in $K$, which are all trivial as $K$ is a solid cylinder, to obtain an element of $c_0(t_i)$ that avoids $B$, as desired.
	\end{proof}

	Denote by $\cS^{\ins}_\sphere$ the set of inward spherical singularities of $\cM$, and by $\cS^{\ins}_\neck$ the set of inward neck  singularities of $\cM$. Similarly, we define $\cS^{\out}_\sphere$ and $\cS^{\out}_\neck$. Then, we denote by $S^{\ins}_\sphere(t)\subset \R^3$ the slice of $\cS^{\ins}_\sphere$ at time $t$, and proceed similarly for the other three sets.
	
	\begin{lem}\label{lem_S_in_S_out_compact}
		$S_{\neck}^\ins(T)$ and $S_{\neck}^\out(T)$ are compact sets.    
	\end{lem}
	
	\begin{proof}
		We only show $S_{\neck}^\ins(T)$ is compact and the proof for $S_{\neck}^\out(T)$ is the same. It suffices to show $\overline{S_{\neck}^\ins(T)}=S_{\neck}^\ins(T)$. By the semi-continuity of the Gaussian density, a limit point $p$ of $S_{\neck}^\ins(T)$ must be a neck singularity. Hence it suffices to show $p\in S_{\neck}^\ins(T)$. We prove by contradiction: Suppose not, then $p\in S_{\neck}^\out(T)$, and by mean convex neighborhood theorem, there is a neighborhood $U$ of $p$ and $\delta>0$ such that the MCF $\{M_t\}_{t\in[T-\delta,T+\delta]}$ in $U$ moves outward. This contradicts the assumption that $p$ is a limit point of $S_{\neck}^\ins(T)$.
	\end{proof}

	\begin{prop} \label{prop_break_in_sing_set} Suppose $c_0\in H_1( W_\ins(T_0) )$ terminates at some time $T>T_0$. Then $c_0$ breaks in $(S^\ins_\neck(T),T)$. The outward case is analogous.
	\end{prop}
	\begin{proof}
		We will only prove the inward case, as the outward case follows analogously. Suppose the otherwise: There exist a neighborhood $U$ of $S^\ins_\neck(T)$ in $\R^3$, an increasing sequence of times $t_1,t_2,...\uparrow T$, and elements $\gamma_i\in c_0(t_i)$ such that each $\gamma_i$ is disjoint from $U$.

		By the mean convex neighborhood theorem and the compactness of $S^\ins_\neck(T)$ and $S^\out_\neck(T)$ from Lemma \ref{lem_S_in_S_out_compact}, we can further pick open neighborhoods $U_\ins, \tilde U_\ins$ with
		$$S^\ins_\neck(T)\subset U_\ins \subset\subset \tilde U_\ins \subset\subset  U,$$
		an open neighborhood $U_\out$ of $S^{\out}_\neck(T)$, and two times $T_1<T<T_2$ such that:
		\begin{itemize}
			\item $\tilde U_\ins$ and $U_\out$ are disjoint.
			\item In the time interval  $(T_1,T_2)$,
			$M(t)\cap \tilde U_\ins$ {\it evolves inward} (i.e.  
			$$K_\ins(t_2)\cap \tilde U_\ins\subset K_\ins(t_1)\backslash M(t)$$
			for every $T_1 <t_1<t_2<T_2$)
			while $M(t)\cap  U_\out$ {\it evolves outward}. 
		\end{itemize}
		
		By Huisken's analysis of spherical singularities (see also the special case of \cite[Theorem 4.6]{ColdingMinicozzi16_SingularSet}), each spherical singularity is isolated in spacetime. Therefore, the limit points of spherical singularities can only be cylindrical singularities.
		
		We claim that after appropriately shrinking the time interval $[T_1,T_2]$, $$ (\R^3\backslash (U_{\ins}\cup U_{\out}))\x [T_1,T_2]$$ has only finitely many singular points, and we can thus assume such singular points all are spherical singularities at time $T$. In fact, suppose not, there exists a sequence of distinct singular points $\{p_i\}_{i=1}^\infty$ outside $ U_{\ins}\cup U_{\out}$, with singular time $t_i\to T$. Then by the compactness of the singular set of $\cM$ and the previous paragraph, there is a subsequence converging to a cylindrical singularity in $(S^{\ins}_\neck(T)\cup S^{\out}_\neck(T))\x\{T\}$. This contradicts our choice of the $p_i$'s. 
		
		As a consequence of the claim, by shrinking $[T_1,T_2]$ and the neighborhoods $\tilde U_\ins$ and  $U_\out$, we can assume $$\overline{\tilde U_\ins\backslash U_\ins}\x[T_1,T_2]$$ consists only of smooth points. Furthermore,
		we can    
		choose a neighborhood $V_\ins$ 
		of $S^\ins_\sphere(T)\backslash \overline{\tilde U_{\ins}}$ such that $M(t)\cap V_\ins$ is a finite union of convex smooth spheres for each $t\in [T_1,T_2]$, using what we proved in the previous paragraph. Similarly, we can find a neighborhood $V_\out$ for $S^\out_\sphere(T)\backslash \overline{\tilde U_{\out}}$ with analogous properties. We can assume the closures of $\tilde U_\ins,U_\out, V_\ins, V_\out$ are all disjoint. Moreover, $M(t)\backslash (U_\ins\cup U_\out\cup V_\ins\cup V_\out)$ evolves smoothly for $t\in [T_1,T_2]$.

		To derive a contradiction to  $\ft(c_0)=T$, we are going to prove that for some $t_i$, there exists a smooth deformation of $\gamma_i$, $\{\gamma^t\subset W_\ins(t)\}_{t\in[t_i,T]}$, with $\gamma^{t_i}=\gamma_i$, thereby letting $\gamma_i$ ``survive" up to  time $T$. Note that:
		
		\begin{itemize}
			\item By the smoothness of $M(t)$ in $\overline{\tilde U_\ins\backslash U_\ins}$ for $t\in [T_1,T_2]$, $$C:=\sup_{t\in[T_1,T_2],\;x\in M(t)\cap \overline{\tilde U_\ins\backslash U_\ins}}|A|<\infty.$$ Thus, the velocity of the flow in this spacetime region is bounded by $C$.
			Thus, since $\gamma_i$ avoids $\tilde U_\ins$, we can take a $t_i\in (T_1,T)$ sufficiently close to $T$ such that there is not enough time for any point of $M(t_i)\backslash \tilde U_\ins$ to be pushed into $U_\ins$ by time $T$. 
			
			\item Note that $M(t)$ evolves outward in $\tilde U_\out$ for $t\in[T_1,T_2]$. 
			\item  Since $V_\ins$ and $V_\out$ consists of spheres, we can remove the components of $\gamma_i$ inside the spheres, so we may assume $\gamma_i$ avoids $V_\ins$ and $V_\out$. 
		\end{itemize}
		
		Combining the above observations, we can construct a smooth deformation of $\gamma_i$, $\{\gamma^t\subset W_\ins(t)\}_{t\in[t_i,T]}$, using the evolution of MCF, with $\gamma^{t_i}=\gamma_i$. This contradicts that $\ft(c_0)=T$.
	\end{proof}
	
	Here comes a key theorem which supports that our definition of homology termination and breakage would accurately describe the heuristic phenomenon shown in Figure \ref{fig::a_break}.
	
	\begin{thm} \label{thm_exist_break_point} Suppose $c_0\in H_1(W_\ins(T_0) )$ terminates at some time $T>T_0$. Then $c_0$ breaks at some inward neck singularity $(x,T)$.
		
		The outward case is analogous.
	\end{thm}
	
	Note that such $x$ may be non-unique: Consider a flow that is a thin torus collapsing into a closed curve consisting entirely of neck singularities.
	
	\begin{proof}
		
		We prove the inward case as the outward case is analogous.
		We will prove by contradiction. Suppose that the theorem is false, meaning:
		
		{\bf Assumption ($\star$):} For every inward neck singularity $(x,T)$, there is a neighborhood $U_x$ of $x$ such that it is not true that ``for every time $t<T$ close enough to $T$, every element of $c_0(t)$ intersects $U_x$".

		Applying 
		Theorem \ref{thm_topological_canonical} to each inward neck singularity $(x,T)$, with a constant $\delta_0(x)>0$ such that $B_{\delta_0(x)}(x)\subset U_x$ and an $R>\max\{R_0,100\}$, we obtain constants $\delta(x),\bar\delta(x)>0$ and a set of full measure $J(x)\subset[T-\bar\delta(x),T+\bar\delta(x)]$ satisfying the properties of Theorem \ref{thm_topological_canonical}.

		Since $S^\ins_\neck(T)$ is compact by Lemma \ref{lem_S_in_S_out_compact}, there exist $x_1,...,x_n\in S^\ins_\neck(T)$ such that $$ B_{\delta(x_1)/2}(x_1),...,B_{\delta(x_n)/2}(x_n)$$ cover $S^\ins_\neck(T)$. For simplicity, we  denote those balls by $\frac 12 B_1,...,\frac 12 B_n$, while
		$$B_1:=B_{\delta(x_1)}(x_1),...,B_n:=B_{\delta(x_n)}(x_n)$$
		Since $c_0$ terminates at time $T$, we know that $c_0$ breaks in $(S^\ins_\neck(T),T)$ by Proposition \ref{prop_break_in_sing_set}. Thus, by definition, there exists a time $T_1$ with $\max_i T-\bar \delta(x_i)<T_1<T$ such that for each $t\in [T_1,T)$, every element of $c_0(t)$ intersects $\cup_i\frac 12 B_i$. We can assume $T_1\in \cap_i J(x_i)$ so that $M(T_1)$ is smooth and intersects each $\partial B_i$ transversely by Theorem \ref{thm_topological_canonical} (\ref{thm_item_transverse}).

		\begin{lem}\label{lem_linking_number_zero}
			Let $D$ be a connected component of $K_\ins(T_1)\cap \partial B_i$ (of which there are at most two according to Theorem \ref{thm_topological_canonical} (\ref{thm_item_at_most_two_convex_disks})), and $\gamma\in c_0(T_1)$. Then, it follows that the linking number $\link(\gamma,\partial D)=0$.
		\end{lem}
		\begin{proof}
			Suppose the otherwise, that there exists some $D$ as above and $\gamma\in c_0(t_0)$ such that  $\link(\gamma,\partial D)\ne 0$. Now, pick any $t_1\in [T_1,T)$ and $\gamma_1\in c_0(t_1)$. By definition, $\gamma_1$ is homologous to $\gamma$ within $W_{\ins}[T_1,t_1]$. Thus, $\gamma_1$ is homologous to $\gamma$ within $\R^3\backslash \partial D$, as the mean convex neighborhood property (Theorem \ref{thm_topological_canonical} (\ref{thm_item_convex})) implies that $\partial D\subset \R^3\backslash W_\ins(t)$ for all $t\in [T_1,t_1]$. Therefore, $\link(\gamma_1,\partial D)\ne 0$, which implies that $\gamma_1$ must intersect $D$. However, since $D\subset \bar B_i\subset U_{x_i}$, this implies that for all $t_1\in [T_1,T)$, any element of $c_0(t_1)$ must intersect $U_{x_i}$. This contradicts the assumption ($\star$).
		\end{proof}
		
		Let $\epsilon_1:=\min_i\delta(x_i)/2$.
		Let $\gamma\in c_0(T_1)$  be such that
		\begin{equation}\label{eq_gamma_def}
			\length(\gamma)<\inf_{\gamma'\in c_0(T_1)}\length(\gamma')+\epsilon_1/100
		\end{equation}
		Without loss of generality, we can assume $\gamma$ intersects all $\partial B_i$ transversely.
		To finish the proof, it suffices to show that $\gamma$ avoids $\cup_i \frac 12 B_i$: This would contradict the definition of $T_1$.
		\begin{lem}
			$\gamma$ does not intersect $\cup_i \frac 12 B_i$.
		\end{lem}
		\begin{proof}
			We prove by contradiction. Suppose that $\gamma$ intersects some $\frac 12 B_i$. We will produce an element of $c_0(T_1)$ whose length is too small.
			
			Without loss of generality, we can assume that no connected component of $\gamma\cap B_i$ is a closed loop. This is because we could just remove all such loops from $\gamma$, and the resulting curve is still in $c_0(T_1)$ by Theorem \ref{thm_topological_canonical} (\ref{thm_item_each_cnn_comp_genus0}).
			Hence, letting $\beta$ be a connected component of $\gamma\cap B_i$, we can assume that $\beta$ is a line segment. 
			
			Now, by Theorem \ref{thm_topological_canonical} (\ref{thm_item_at_most_two_convex_disks}) and our choice that $T_1\in\cap_i J(x_i)$, $W_{\ins}(T_1)\cap \partial B_i$ consists of at most two disks.  There are two cases: Either (1) $\beta$ starts and ends on the same disk, say $D_1$, or (2) $\beta$ starts and ends on different disks, $D_1$ and $D_2$. We will show that both are impossible.
			
			For case (1), since $\beta$ intersects $\frac 12 B_i$, whose  distance to $\partial B_i$ is  $\delta(x_i)/2$, we know that $\length(\beta)$ is at least $\delta(x_i)$. On the other hand, note that by Theorem \ref{thm_topological_canonical} (\ref{thm_item_is_colid_cylinder}), (\ref{thm_item_convex}), and (\ref{thm_item_convex_loop}), $D_1$ is a convex disc on $\partial B_i$ with diameter less than $\delta(x_i)/50$ (recall $R>100$). Thus,
			we can join the end points of $\beta$, from $\beta(1)$ to $\beta(0)$, by a segment $\beta_1$ on $D_1$ of length less than $\delta(x_i)/50$: See Figure \ref{fig::shorten1}. Then, we consider the new loop $\gamma-\beta-\beta'$, which replaces $\beta\subset\gamma$ with $\beta'$. This loop lies in $c_0(T_1)$, because $\beta+\beta'$ bounds a disc in $W_\ins (T_1)\cap \bar B_i$ by Theorem \ref{thm_topological_canonical} (\ref{thm_item_each_cnn_comp_genus0}).  
			
			\begin{figure}[h]
				\centering
				\makebox[\textwidth][c]{\includegraphics[width=3.0in]{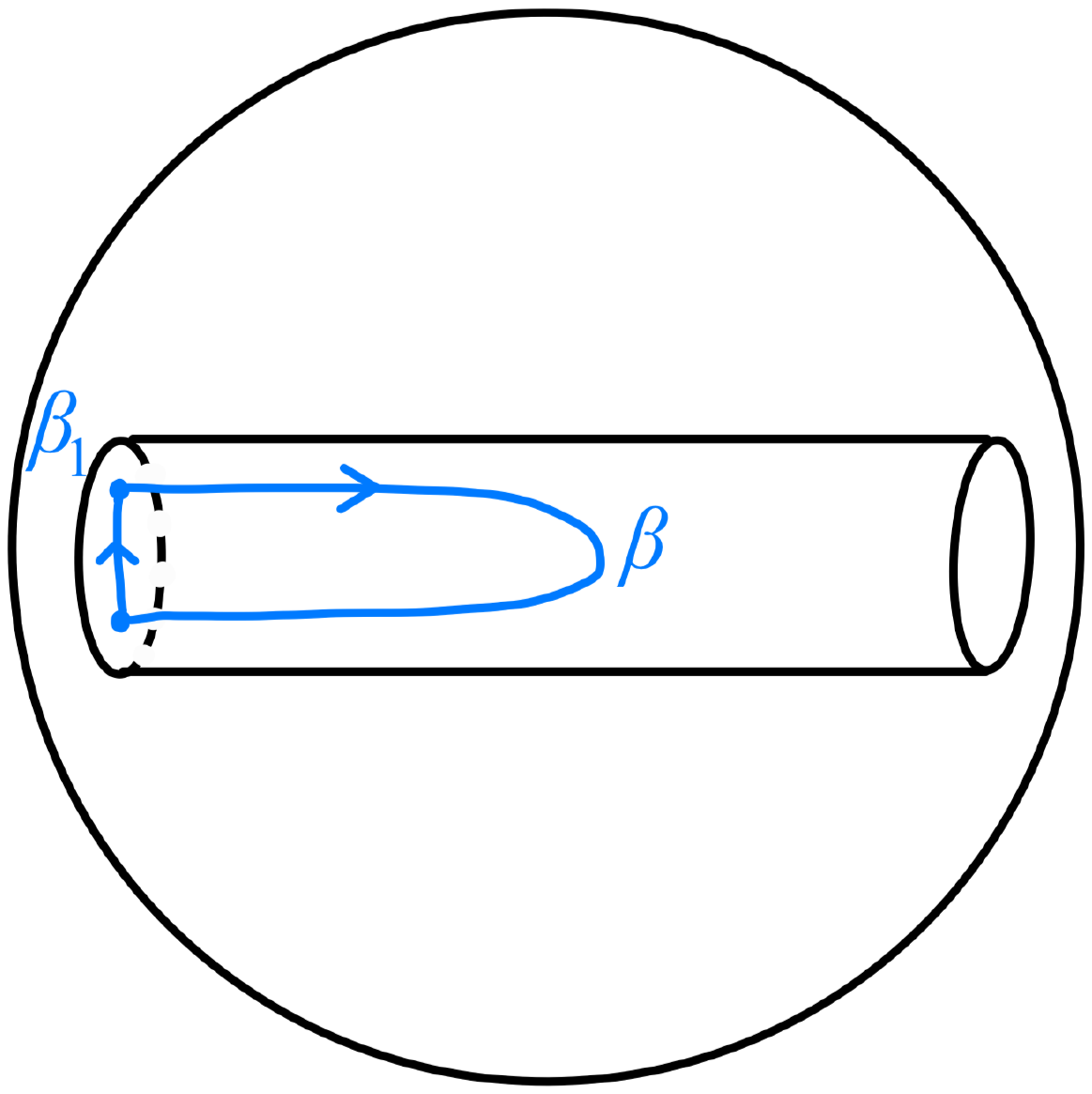}}
				\caption{}
				\label{fig::shorten1}
			\end{figure}
			
			Moreover, this new loop is impossibly short:
			\begin{align*}
				\length(\gamma-\beta-\beta')&\leq \length(\gamma)-\delta(x_i)+ \delta(x_i)/50\\
				&<\length(\gamma)-\delta(x_i)/2\\
				&\leq \length(\gamma)-\epsilon_1\\
				&<\inf_{\gamma'\in c_0(T_1)}\length(\gamma'),
			\end{align*}
			in which the last inequality is from the definition of $\gamma$. Thus, 
			a contradiction arises, and case (1) is impossible.

			For case (2), suppose the starting point $\beta(0)$ is in $D_1$ and the ending point $\beta(1)$ is in $ D_2$. We claim that there is another connected component $\hat\beta$ of $\gamma\cap B_i$ such that  starting point  $\hat\beta(0)$ is in $ D_2$ and ending point $\hat\beta(1)$ is in $\in D_1$. This claim follows immediately from:
			\begin{itemize}
				\item By Theorem \ref{thm_topological_canonical} (\ref{thm_item_each_cnn_comp_genus0}), $M(T_1)\cap \partial B_i$ is a cylinder.
				\item By 
				Lemma \ref{lem_linking_number_zero}, $\link(\gamma,\partial D_1)=\link(\gamma,\partial D_2)=0$.
				\item Case (1) was proven impossible.
			\end{itemize}
			Finally, let $\beta_1$ be a segment on $D_1$ connecting $\hat\beta(1)$ to $\beta(0)$, and $\beta_2$ be a segment on $D_2$ connecting   $\hat\beta(0)$ to $\beta(1)$ (see Figure \ref{fig::shorten2}). As in case (1), we can guarantee  $\length(\beta_1),\length(\beta_2)<\delta(x_i)/50$. Hence,  we  consider the new loop $\gamma-\beta-\hat\beta-\beta_1-\beta_2$, which replaces $\beta+\hat\beta\subset \gamma$ with $-\beta_1-\beta_2$. This new loop lies in $c_0(T_1)$, because $\beta+\hat\beta+\beta_1+\beta_2$ bounds a disc in $W_\ins (T_1)\cap\bar B_i$ by Theorem \ref{thm_topological_canonical} (\ref{thm_item_each_cnn_comp_genus0}). Moreover, as in case (1), we can show that 
			$$\length(\gamma-\beta-\hat\beta-\beta_1-\beta_2)<\inf_{\gamma'\in c_0(T_1)}\length(\gamma'),$$
			which is a contradiction. Therefore, case (2) is also impossible. This leads to a contradiction.
			
			\begin{figure}[h!]
				\centering
				\makebox[\textwidth][c]{\includegraphics[width=2.5in]{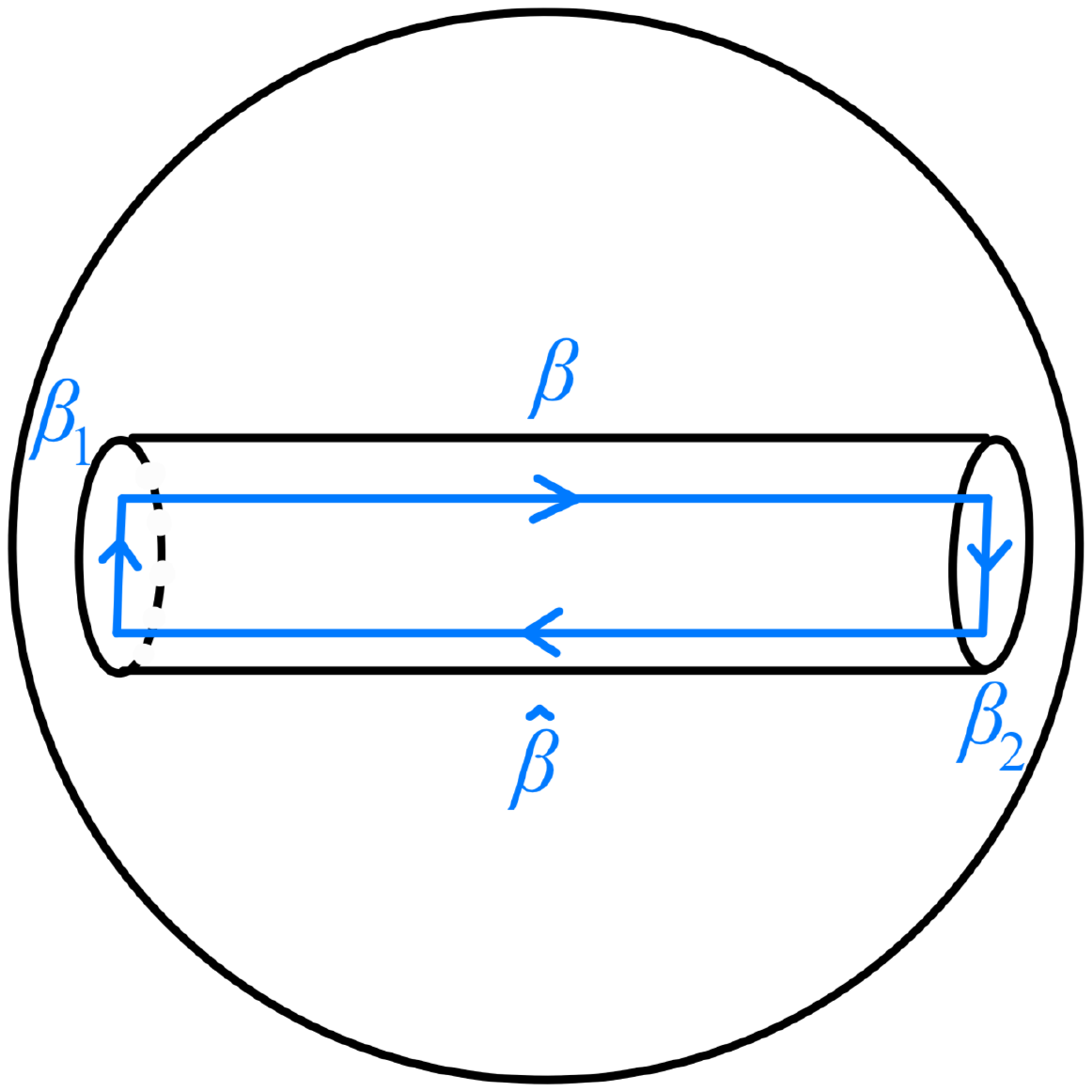}}
				\caption{}
				\label{fig::shorten2}
			\end{figure}

		\end{proof}
		
		This finishes the proof of Theorem \ref{thm_exist_break_point}.
	\end{proof}
	
	\subsection{MCF through cylindrical and spherical singularities from torus}\label{sect_MCF_torus} 
	In \S \ref{sect_MCF_torus} , we will focus on  2-dimensional MCF $\cM=\{M(t)\}_{t\geq 0}$ through  cylindrical and spherical singularities in $\R^3$, where $M(0)$ is a smooth {\it torus}.
	The main goal of \S \ref{sect_MCF_torus} is to prove the following.  
	\begin{thm}\label{thm_terminate_time_change_genus} The setting is as follows.
		\begin{itemize}
			\item Let $\{M(t)\}_{t\geq 0}$ be a MCF through cylindrical and spherical singularities with $M(0)$ a smooth torus in $\R^3$.
			\item Let $a_0$ be a generator of $ H_1(W_\ins(0))\cong\Z$, and   $b_0$ be a generator of $H_1(W_\out(0))\cong \Z$.
			\item Let 
			$T=\min\{\ft(a_0),\ft(b_0)\}.$
		\end{itemize}
		Then $T<\infty$, and $\genus(M(t))=1$ for a.e. $t<T$, while $\genus(M(t))=0$ or $M(t)$ is empty for a.e. $t>T$.
	\end{thm}
	Throughout \S \ref{sect_MCF_torus}, we will retain the notations in this theorem. 
	
	Let us first sketch the proof.  By \cite{ColdingMinicozzi16_SingularSet}, $M(t)$ is smooth for a.e. time. And by \cite{White95_WSF_Top}, $\genus(M(t))$, when well-defined, is non-increasing in $t$. Thus, there exists some time $T_g$ such that $\genus(M(t))=1$ for a.e. $t<T_g$, while $\genus(M(t))=0$ or $M(t)$ is empty for a.e. $t>T_g$. Our goal is to show $T=T_g$. 
	
	The proof consists of proving the following six claims one-by-one:
	\begin{itemize}
		\item $T<\infty$.
		\item Let $t\geq 0$. If $M(t)$ is a smooth torus and  $a_0(t)$ exists, then $a_0(t)$ generates $H_1(W_\ins(t))$. And the case for $b_0$ is analogous.\
		\item $T_g\geq T$.
		\item $\ft(a_0)\ne \ft(b_0)$.
		\item If $\ft(a_0)< \ft(b_0)$, then $b_0(t)$ is trivial for each $t> \ft(a_0)$. And if $\ft(b_0)< \ft(a_0)$, then $a_0(t)$ is trivial for each $t> \ft(b_0)$.
		\item $T_g\leq T$.
	\end{itemize}
	
	We now begin the proof.
	\begin{prop}\label{prop_will_terminate}
		$T<\infty$.
	\end{prop}
	\begin{proof}
		Suppose otherwise, i.e. $a_0$ and $b_0$ both never terminate. Since $M(0)$ is compact, eventually $K_\out(t)=\R^3$. So $a_0(T)$ and $b_0(T)$ both become trivial for some large $T>0$. As a result, if we pick some loops $\alpha_0\in a_0$ and $\beta_0\in b_0$, then there exist 2-chain  $A\subset W_\ins [0,T]$ and $B\subset W_\out[0,T]$ such that $\partial A=\alpha_0$ and $\partial B=\beta_0$.
		
		Now, denote by $\hat B\subset \R^3\x[-T,0]$ the reflection of  $B$ across $\R^3\x\{0\}$. Let $\tilde B=B\cup\hat B$, which can be viewed as a {\it closed} 2-chain in $\R^4$.
		Then we view $A\subset \R^4\backslash\tilde B$.
		Thus, to derive a contradiction, it suffices to show that $\alpha_0$ is homologically non-trivial in $\R^4\backslash \tilde B$. 
		
		Without loss of generality, we can assume $\tilde B$ is connected by discarding all those connected components that do not contain $\beta_0$. By Alexander duality, $$H_1(\R^4\backslash \tilde B)\cong H^2(\tilde B)\cong \Z.$$
		One can check that $\alpha_0\subset \R^4\backslash \tilde B $ actually generates  $\Z$ as the linking number $\link(a_0,b_0)=1$. This shows $\alpha_0$ is homologically non-trivial in $\R^4\backslash \tilde B$, contradicting the existence of $A$.
	\end{proof}
	
	\begin{rmk}\label{rmk_any_genus}
		The above proof works also in the case when $M(0)$ is a closed surface of any genus with $a_0\in H_1(W_\ins(0))$ and $b_0\in H_1(W_\out(0))$ linked, and the flow $\{M(t)\}_{t\geq 0}$ is a general level set flow (whose singularities are not necessarily cylindrical or spherical).
	\end{rmk}

	\begin{prop}\label{prop_if_exists_then_generator} Let $t\geq 0$. If $M(t)$ is a smooth torus and  $a_0(t)$ exists, then $a_0(t)$ generates $H_1(W_\ins(t))$. And the case for $b_0$ is analogous.
	\end{prop}
	\begin{proof} We will just prove the case for $a_0$.
		Let $\bar a$ be a generator of $H_1(W_\ins(t))\cong \Z$. It suffices to show $\bar a=a_0(t)$ up to a sign.
		
		By definition, there exists $\alpha_0\in a_0$, $\alpha_1\in a_0(t)$ such that $\alpha_0-\alpha_1=\partial A$ for some $A\subset W[0,t]$. On the other hand, pick a loop $\bar \alpha_1\in\bar a$, then by \cite[Theorem 1 (ii)]{White95_WSF_Top}, there exists a homotopy $H$ in $W[0,T]$ joining $\bar \alpha_1$ back to some loop $\bar \alpha_0\subset W(0)$ (which means $\partial H=\bar \alpha_1-\bar \alpha_0$). So $[\bar\alpha_0]=k a_0$ for some integer $k$, and so $\bar\alpha_0-k\alpha_0=\partial A_0$ for some $A_0\subset W(0)$. If we manage to show 
		$a_0=[\bar\alpha_0]$ or $-[\bar\alpha_0]$, then by the fact that $a_0$ can only descend into one class at time $t$ (Proposition \ref{prop_homology_unique}), we would know $a_0(t)=\bar a$ or $-\bar a$, as desired. Hence, it suffices to show that $k=\pm 1$.

		Let us glue $H, A_0,$ and $kA$ together, so that we have
		$$\bar\alpha_1-k\alpha_1=\partial(H+A_0+kA).$$
		Thus, since the inclusion $H_1(W_\ins(t))\to H_1(W_\ins[0,t])$ is injective  by \cite[Theorem 1 (iii)]{White95_WSF_Top}, $\bar a=k\alpha_0(t)$ in $H_1(W_\ins(t))$. Since $\bar a$ is a generator by definition, $k=\pm 1$, as desired.
	\end{proof}
	
	\begin{prop}\label{prop_T_g_geq_T}
		$T_g\geq T$.
	\end{prop}
	\begin{proof}
		Let us assume $T=\ft(a_0)$, as the other case $T=\ft(b_0)$ is analogous. Recall that we have shown $T<\infty$. Since $\genus(M(t))$, if well-defined, is non-increasing in $t$, it suffices to prove that there exists $T_1<T$ such that  for a dense set of $t\in (T_1,T)$, $\genus(M(t))=1$.
		
		By Theorem \ref{thm_exist_break_point}, $T=\ft(a_0)$ implies $a_0$ breaks at some inward neck singularity $(x,T)$. Then, applying   Proposition \ref{prop_shape_of_breaking_neck} to $(x,T)$ with $\delta_0=1$ and an $R>R_0$,
		we obtain constants $\delta,\bar\delta$ and a dense set $J\subset [T-\bar\delta,T+\bar \delta]$ with $T-\bar\delta\in J$. We let $T_1=T-\bar\delta$, and  $B=B_\delta(x)$.
		
		Now, fix any $t\in (T_1,T)$, and $D$ let be one of the two connected component of $K_\ins(t)\cap \partial B$: Recall that $K_\ins(t)\cap B$ is a solid cylinder by Proposition \ref{prop_shape_of_breaking_neck}. By Proposition \ref{prop_shape_of_breaking_neck}, some element $\alpha\in a_0(t)$ has a non-zero intersection number with $D$. Now, we push $\partial D$ slightly into $K_\out(t) \cap B$ and call that loop $\beta$. Then the linking number $\link(\beta,\alpha)$ is non-zero, with $\alpha$ inside $M(t)$ and $\beta$ outside $M(t)$. Hence, $\genus(M(t))$ is non-zero, and thus has to be one, as desired.
	\end{proof}

	\begin{prop}\label{prop_cannot_terminate_same_time}
		$\ft(a_0)\ne \ft(b_0)$.
	\end{prop}
	\begin{proof}
		If $\ft(b_0)<\ft(a_0)$, we are done. So let us assume $\ft(a_0)\leq \ft(b_0)$ and aim to show $\ft(b_0)>\ft(a_0)$.
		
		Let us focus on the time $t=T_1$, with $T_1:=T-\bar\delta$, as defined in the proof of Proposition \ref{prop_T_g_geq_T}. We know $\genus(M(T_1))=1$ from before. Now, consider the loops $\alpha\in a_0(T_1)$ and $\beta\subset W_\out(T_1)\cap B$ defined in the previous proof. Then by Proposition \ref{prop_if_exists_then_generator}, $\alpha$ is a generator of $H_1(W_\ins(T_1))$, and from the construction of $\beta$ it is clear  $\link(\beta,\alpha)=\pm 1$. So $\beta$ actually generates $H_1(W_\out(T_1))$. Then by Proposition \ref{prop_if_exists_then_generator} again and the assumption $\ft(b_0)\geq \ft(a_0)$, we have $[\beta]=b_0(T_1)$, possibly after changing the orientation of $\beta$. 
		
		Finally, by the mean convex neighborhood property, $\beta\subset W_\out(T_1)\cap B$ will survive after time $T$. So $\ft(b_0)>\ft(a_0)$.	
	\end{proof}
	
	\begin{prop}\label{prop_only_one_terminate}
		If $\ft(a_0)< \ft(b_0)$, then $b_0(t)$ exists and is trivial for each $t> \ft(a_0)$. And if $\ft(b_0)< \ft(a_0)$, then $a_0(t)$ exists and is trivial for each $t> \ft(b_0)$.
	\end{prop}
	\begin{proof}
		We prove the first statement and the second statement is similar. Let us retain the notation from the previous proof. By Proposition \ref{prop_shape_of_breaking_neck}, $M(T_1)\cap B$ (recall $T_1=T-\bar\delta$) is close to a round cylinder. Now, enclose this cylinder by an Angenent torus,  and run the MCF. Note that:
		\begin{itemize}
			\item Since the time interval around $T$ given by the mean convex neighborhood property is independent of $R$ (in Proposition \ref{prop_shape_of_breaking_neck}), we can, by making $R$ very large and thus the Angenent torus very small, assume that the mean convex neighborhood property still holds at the moment the Angenent torus vanishes.
			\item By the avoidance principle, the distance between the Angenent torus and $M(t)$ is non-decreasing.
		\end{itemize}
		Hence, when the Angenent torus vanishes, the neck $M(t)\cap B$ has already been ``cut into disconnected pieces.'' As a result, the loop $\beta$, which remains disjoint from the evolving surface, would have become trivial at the moment the Angenent torus disappears.
		
		Finally, note that as $R\to\infty$, $\bar\delta\to 0$ (see Theorem \ref{thm_topological_canonical}, item \ref{thm_item_is_colid_cylinder}). By the definition of cylindrical singularity, we know $T_1=T-\bar\delta\to T$ and $M(T-\bar\delta)\cap B$ tends to be an actual round cylinder after rescaling by the factor $R$. This shows that the moment when Angenent torus vanishes will tend to $T$. Therefore, $b_0(t)$ is trivial for each $t>T$. 
	\end{proof}

	Finally, since we have already proven $T_g\geq T$, to complete the proof of Theorem \ref{thm_terminate_time_change_genus}, it remains to show:
	\begin{prop}
		$T_g\leq T$.
	\end{prop}
	\begin{proof}
		Suppose by contradiction $T_g>T$. Again, we can just consider the case $\ft(a_0)< \ft(b_0)$. By our Proposition \ref{prop_only_one_terminate}, we can pick a time $T_2\in (T,T_g)$ when $M(T_2)$ is a smooth torus and $b_0(T_2)$ exists and is trivial. This contradicts Proposition \ref{prop_if_exists_then_generator}, which says that $b_0(T_2)$ generates $H_1(W_\out(T_2))$.
	\end{proof}

	This completes the proof of Theorem \ref{thm_terminate_time_change_genus}.
	\subsection{Termination time of limit of MCF}\label{sect_termination_time}
	Finally, in \S \ref{sect_termination_time}, let us mention a proposition that describes a relationship between the termination time and a convergent sequence of initial conditions.
	\begin{prop}\label{prop_termination_time_liminf}
		The setting is as follows.
		\begin{itemize}
			\item Let  $\cM^i=\{M^i(t)\}_{t\geq 0}$, $i=1,2,...$, and $\cM=\{M(t)\}_{t\geq 0}$ all be  MCF through cylindrical and spherical singularities, such that each $M^i(0)$ and $M(0)$ are smooth, close hypersurfaces.
			\item For each $i$, assume $M^i(0)$ is sufficiently close in $C^\infty$ to $M(0)$ such that each $H_1(W^{\cM^i}(0))$ can be canonically identified with $H_1(W^{\cM}(0))$. Moreover, $M^i(0)\to M(0)$ in $C^\infty$.
			\item Let $c_0\in H_1(W^{\cM}(0) )$. Note that $c_0$ can be viewed as an element of $H_1(W^{\cM^i}(0))$ for each $i$ too.
		\end{itemize}
		Then $$\liminf_i\ft^{\cM^i}(c_0)\geq \ft^{\cM}(c_0).$$
	\end{prop}
	\begin{proof}
		Let $T=\ft^{\cM}(c_0)$.
		
		We first consider the case $T<\infty$. Suppose by contradiction that there exists a subsequence $\{i_k\}_{k=1}^\infty$ and some $T_1<T$ such that $\ft^{\cM^{i_k}}(c_0)\leq T_1$ for each $k$.  Pick some element $\gamma_0\subset W^{\cM}(0)$ with $[\gamma_0]=c_0$, and  $\gamma_1\subset W^{\cM}(\frac{T_1+T}2)$ with $[\gamma_1]=c_0(\frac{T_1+T}2)$. By definition, $\gamma_0$ and $\gamma_1$ together bound some $\Gamma\subset W^{\cM}[0,\frac{T_1+T}2]$. 
		
		Now, recall that $\cM^i\to\cM$ in the Hausdorff sense by Proposition \ref{lem_ae_time_smooth_conv}. Thus, since $\Gamma$ is compact, for all sufficiently large $i$, $\Gamma\subset W^{\cM^i}[0,\frac{T_1+T}2]$. Moreover, $\gamma_0$  represents $c_0\in H_1(W^{\cM^i}(0))$ for such large $i$. This contradicts that $\ft^{\cM^{i_k}}(c_0)\leq T_1$ for each $k$.
		
		Lastly,  the case $T=\infty$ can be done similarly using the fact that the flow $\cM$ vanishes in finite time.
	\end{proof}

	\section{Proof of main theorems}\label{sect_proof}
	\subsection{Proof of Theorem \ref{thm:main1}} Suppose by contradiction that for each $s\in [0,1]$, $\{M^s(t)\}_{t\geq 0}$ is a MCF through cylindrical and spherical singularities.
	For each $s\in [0,1]$, let $$T^s=\min\{\ft^{\cM^s}(a_0),\ft^{\cM^s}(b_0)\}.$$ Furthermore, Propositions \ref{prop_cannot_terminate_same_time} and \ref{prop_only_one_terminate} show that either $a_0$ or $b_0$ will terminate, but not both. Consequently, we can represent $[0,1]$ as a disjoint union $A\sqcup B$, where $A$ contains those $s$ for which $T^s=\ft^{\cM^s}(a_0)$, and $B$ contains those $s$ for which $T^s=\ft^{\cM^s}(b_0)$. Note that $0\in A$ and $1\in B$ by the assumption. Thus, the following lemma leads us directly to a contradiction.
	
	\begin{lem}
		The sets $A$ and $B$ are both closed.
	\end{lem}
	\begin{proof}
		We will just prove that $A$ is closed.
		Let $s\in [0,1]$ be an accumulation point of $A$, and pick a sequence $s_i$ in $A$ with $s_i\to s$. Note that:
		\begin{itemize}
			\item For each $i$, by Theorem \ref{thm_terminate_time_change_genus}, $\genus(M^{s_i}(t))=1$ for a.e. $t<T^{s_i}$ and $\genus(M^{s_i}(t))=0$ for a.e. $t>T^{s_i}$. 
			\item Similarly, $\genus(M^s(t))=1$ for a.e. $t<T^{s}$ and $\genus(M^{s}(t))=0$ for a.e. $t>T^{s}$. 
		\end{itemize} 
		Thus, together with Proposition \ref{lem_ae_time_smooth_conv}, which says $M^s_i(t)\to M^s(t)$ in $C^\infty$ for a.e. $t\geq 0$, we know $T^{s_i}\to T^s$. Hence,
		$$T^s=\liminf_i  T^{s_i}=\liminf_i \ft^{\cM^{s_i}}(a_0)\geq \ft^{\cM^{s}}(a_0).$$
		Note that the second equality holds because $s_i\in A$, and the inequality holds by Proposition \ref{prop_termination_time_liminf}. Thus,  we know $T^s=\ft^{\cM^{s}}(a_0)$, which means for the flow $\cM^s$,  $a_0$ will terminate but  $b_0$ will not. So $s\in A$. This shows that $A$ is closed.
	\end{proof}
	This finishes the proof of Theorem \ref{thm:main1}.
	\begin{rmk}\label{rmk_main_thm_fail}
		Let us explain why Theorem \ref{thm:main1} would not hold if the initial conditions have genus greater than one. For example, consider the genus $2$ surface depicted in    
		Figure \ref{fig:genus_two_b}, where $a_0$ and $b_0$ are linked  as shown. Then, the MCF actually could develop both inward and outward cylindrical singularities simultaneously, with $a_0$ breaking at the inward one and $b_0$ breaking at the outward one. This phenomenon may prevent a genus one singularity from showing up in any intermediate flow between $\{M^0(t)\}_{t\geq 0}$ and $\{M^1(t)\}_{t\geq 0}$, in the setting of Theorem \ref{thm:main1}.
		
		One might think if we choose $a_0$ and $b_0$ better, like in Figure \ref{fig:genus_two_a}, then the conclusion of Theorem \ref{thm:main1} may hold. However, Figure \ref{fig:genus_two_b} and \ref{fig:genus_two_a} are actually homotopic to each other. In conclusion, in a genus two surface, we cannot force a genus one singularity to appear just by topology: The geometry of the initial conditions must play a role.
	\end{rmk}

	\begin{figure}[h]
		\centering
		\begin{minipage}{0.45\textwidth}
			\centering
			\includegraphics[width=1.2\textwidth]{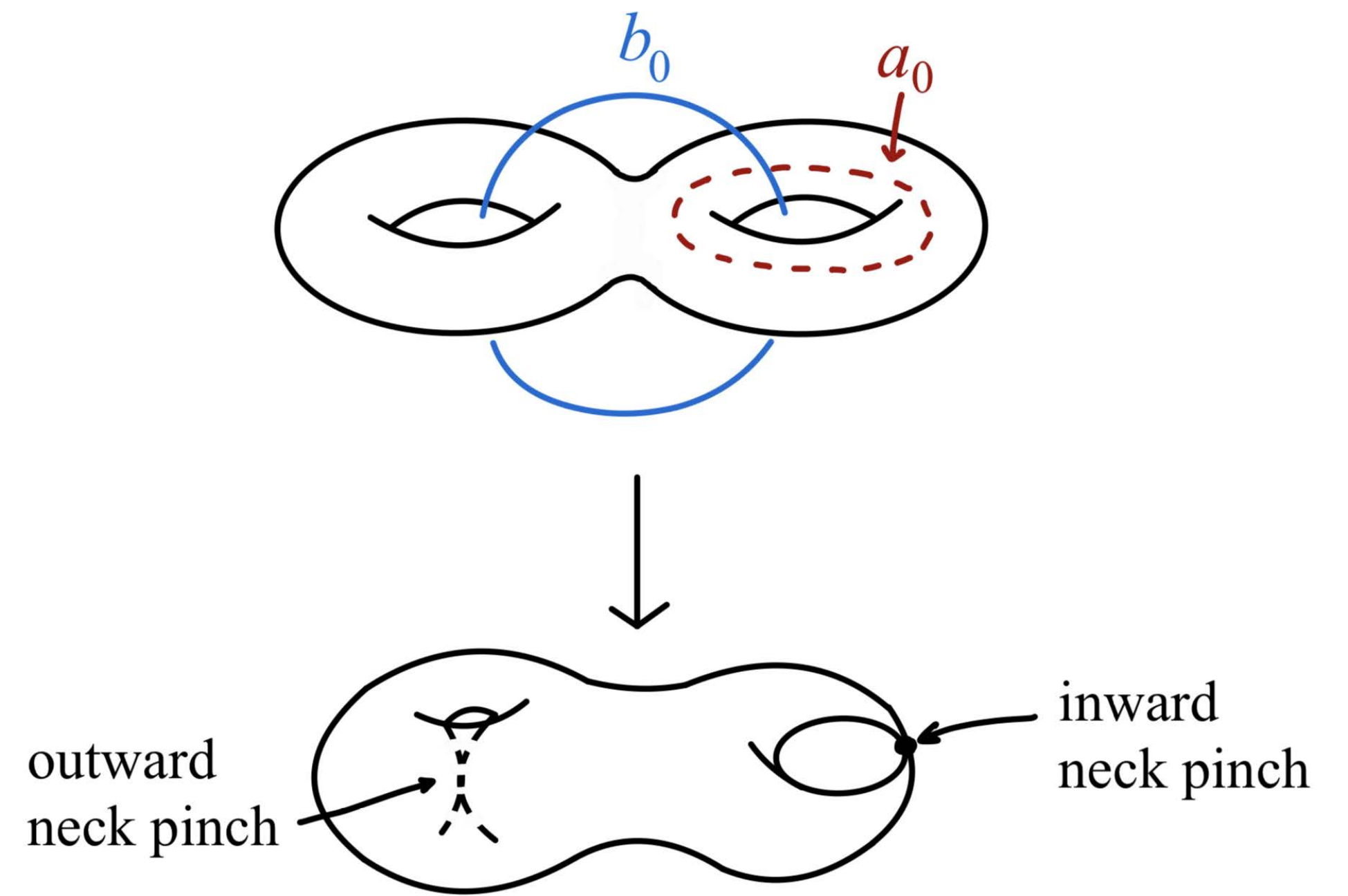}% second figure itself
			\caption{}\label{fig:genus_two_b}
		\end{minipage}
		\begin{minipage}{0.45\textwidth}
			\centering
			\includegraphics[width=0.7\textwidth]{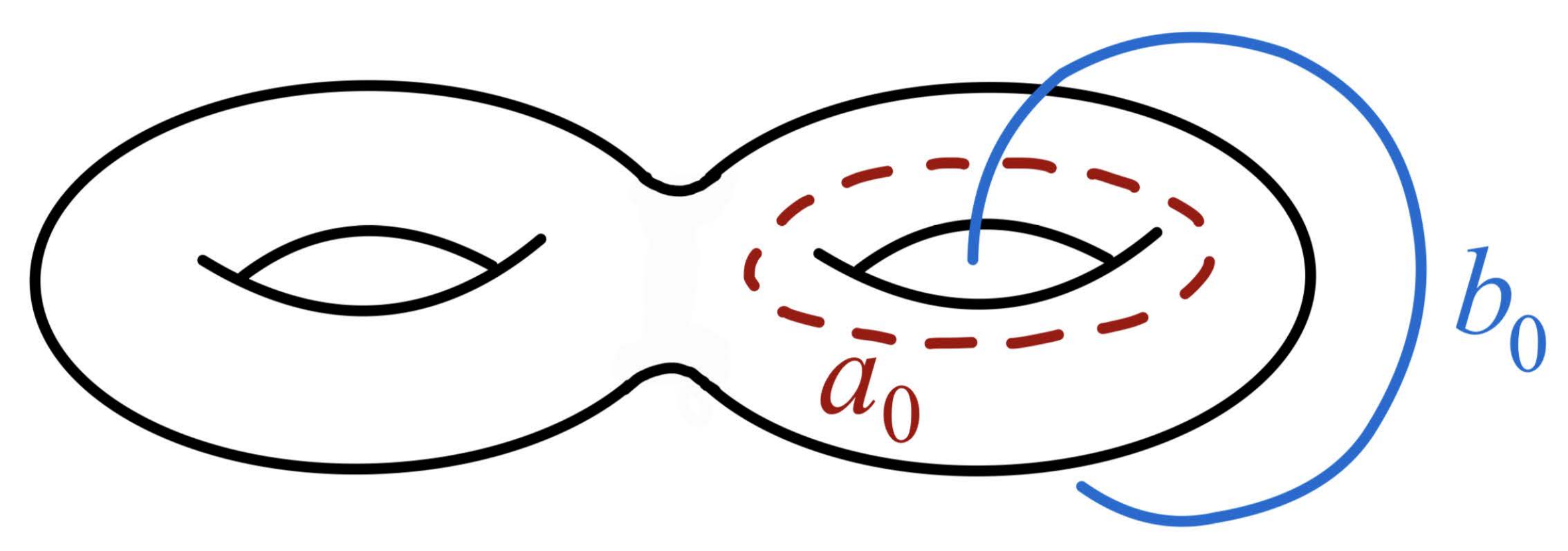} % first figure itself
			\caption{}
			\label{fig:genus_two_a}
		\end{minipage}\hfill
	\end{figure}
	
	\subsection{Proof of Corollary \ref{cor_less_than_2}}
	Let $\cM^s:=\{M^s(t)\}_{t\geq 0}$ be the level set flow starting from $M^s(0):=M^s$. We can apply Theorem \ref{thm:main1} to the flows $\cM^s$, $s\in[0,1]$, which shows there exists $s_0\in [0,1]$ such that $\cM^{s_0}$ has a singularity $(x,T)$ that is not (multiplicity one) cylindrical or spherical. In other words, every tangent flow $\cM'$ at $(x,T)$ is not the shrinking cylinder or sphere of multiplicity one. Recall that by \cite{Ilmanen95_Sing2D}, $\cM'$ is a smooth, embedded, self-shrinking flow $\{\sqrt{-t}m\Sigma'\}_{t<0}$ with genus at most one and has multiplicity $m$. But the multiplicity can only be $1$ by the entropy bound $\Ent(M^{s_0})<2$ and the monotonicity formula. Thus, $\Sigma'$ has genus $1$.
	
	\subsection{Proof of Theorem \ref{thm:main2}}
	Note that we have $\Ent(M^s)<2$ for each $s$ as $M^s$ is close to $\mathbb T$, which has entropy less than $2$.
	To apply Corollary \ref{cor_less_than_2}, it suffices to show that for the MCF starting from $M^0$ (resp. $M^1$), the inward (resp. outward) torus neck will pinch. But this is given by Proposition \ref{prop_perturb_shrinker_which_terminate_first}.
	
	\subsection{Proof of Theorem \ref{thm_genus1_least_entropy}}
	Let $\Sigma_1$ be a genus one embedded shrinker in $\R^3$ with the least entropy. Recall that by \cite{ColdingMinicozzi12_generic} $\index(\Sigma_1)\geq 5$. Therefore, in order to prove Theorem \ref{thm_genus1_least_entropy}, let us suppose by contradiction that $\Sigma_1$ is compact with index at least $6$.
	
	We first need a family of initial conditions to run MCF. That will be provided by the following lemma.
	\begin{lem}\label{thm_family_initial_condition}
		Let $\Sigma^n$ by any smooth, embedded, compact, $n$-dimensional shrinker in $\R^{n+1}$ with index at least $6$. Let $\epsilon>0$ be sufficiently small.
		Then there exists a one-parameter  family of smooth, compact, embedded surfaces $\{M^s(0)\}_{s\in [0,1]}$ such that:
		\begin{enumerate}
			\item\label{thm_family_initial_condition_cts} The family varies continuous in the $C^\infty$-topology, and each $M^s(0)$ is $\epsilon$-close to $C^\infty$ to $\Sigma$.
			\item\label{thm_family_initial_condition_entropy} Each $M^s(0)$ has entropy less than that of $\Sigma$. 
			\item\label{thm_family_initial_condition_inside} $M^0(0), M^1(0),$ and $\Sigma$ are all disjoint, with $M^0(0)$  inside $\Sigma$ and $M^1(0)$  outside.
		\end{enumerate}
	\end{lem}
	\begin{proof}
		Fix an outward unit normal vector field $\bf n$ to $\Sigma$. Since $\index(\Sigma)\geq 6$, the eigenfunctions of its Jacobi operator, with respect to the Gaussian metric, that have negative eigenvalues include:
		\begin{itemize}
			\item three induced by translation in $\R^3$,
			\item one by scaling,
			\item the unique one-sided one which has the lowest eigenvalue, denoted $\phi_0$,
			\item and at least one more, denoted $\phi_1$,
		\end{itemize}
		all of which are orthonormal under the $L^2$-inner product.
		We will choose $\phi_0>0$.
		
		Let $\epsilon>0$, and define $M^s(0)$ to be the following perturbation of $\Sigma$: $$M^s(0):=\Sigma+\epsilon(-\cos(s\pi)\phi_0+\sin(s\pi)\phi_1)\bf n.$$
		Thus, if $\epsilon>0$ is sufficiently small, clearly the family $\{M^s(0)\}_{s\in [0,1]}$ is smooth. Item (\ref{thm_family_initial_condition_inside}) holds because $\phi_0>0$. Finally, (\ref{thm_family_initial_condition_entropy}) holds because $\phi_0,\phi_1$ are not induced by translation or scaling (see Theorem 0.15 in \cite{ColdingMinicozzi12_generic}).
	\end{proof}
	
	Applying the above lemma to $\Sigma_1$, we obtain a one-parameter family $\{M^s(0)\}_{s\in [0,1]}$ of tori. Then $$\Ent(M^s(0))<\Ent(\Sigma_1)\leq \Ent(\mathbb T)<2.$$ Thus, applying Corollary \ref{cor_less_than_2}, and by the monotonicity formula, we obtain another embedded genus one shrinker with entropy less than $\Sigma_1$, which contradicts the definition of $\Sigma_1$.

	\subsection{Proof of Theorem \ref{thm_eternal}}
	Since $\mathbb T$ is rotationally symmetric, by \cite{Liu2016_index_shrinker}, it has index at least $7$. Again, we need a family of MCF. We will apply \cite[Theorem 1.6]{ChoiMantoulidis22AncientGradientFlows} of Choi-Mantoulidis. Namely, since $\mathbb T$ is a minimal surface with index at least $6$  under the Gaussian metric, it has, as we saw in the proof of Lemma \ref{thm_family_initial_condition}, two orthonormal eigenfunctions $\phi_0,\phi_1$ to the Jacobi operator that
	\begin{itemize}
		\item have negative eigenvalues,
		\item and are both orthogonal to the other 4 eigenfunctions induced by translation and scaling.
	\end{itemize}
	Now, pick an $\epsilon>0$. Applying \cite[Theorem 1.6]{ChoiMantoulidis22AncientGradientFlows} to the 2-dimensional function space spanned by $\phi_0$ and $\phi_1$, we obtain a one-parameter family of smooth ancient {\it rescaled MCF} (i.e. MCF under the Gaussian metric) $\tilde\cM^s=\{\tilde M^s(\tau)\}_{\tau\leq 0}$, $s\in[0,1]$, such that:
	\begin{itemize}
		\item For each $s$, $\tilde M^s(t)\to \mathbb T$ in $C^\infty$ as $t\to-\infty$. 
		\item $\tilde M^0(0)$ lies inside $\mathbb T$ while $\tilde M^1(0)$ lies outside.
		\item $\{\tilde M^s(0)\}_{s\in[0,1]}$ is a smooth family of tori, each being $\epsilon$-close to $\mathbb T$ in $C^\infty$ (see \cite[Corollary 3.4]{ChoiMantoulidis22AncientGradientFlows}).
	\end{itemize}
	
	If $\epsilon$ is small enough, we can apply Theorem \ref{thm:main2}  to the family $\{\tilde M^s(0)\}_{s\in[0,1]}$ to obtain an $s_0\in [0,1]$ such that the level set flow $\{M(t)\}_{t\geq 0}$ with initial condition $M(0)=\tilde M^s(0)$ would develop a singularity at which every tangent flow is given by a multiplicity one, embedded, genus one self-shrinker. 
	
	Finally, we define an ancient smooth MCF $\{\hat M(t)\}_{t\leq -1}$ by rescaling the rescaled MCF $\{\tilde M^{s_0}(\tau)\}_{\tau\leq 0}$:
	$$\hat M(t)=\sqrt{-t}\tilde M(-\log(-t)),\;\;t\leq -1.$$
	Note that $\hat M(-1)=\tilde M(0)=M(0)$. Hence,
	combining the two flows $\{\hat M(t)\}_{t\leq -1}$ and $\{M(t)\}_{t\geq 0}$, we obtain an ancient MCF satisfying Theorem \ref{thm_eternal}.
	
	\subsection{Proof of Corollary \ref{cor_fourth}}
	Let $\Sigma$ be an embedded shrinker with the fourth least entropy in $\R^3$, whose existence was established in \S \ref{sect_intro} already.  
	Suppose by contradiction that $\Sigma$ is rotationally symmetric. Then by Kleene-M\o ller \cite{KM_rotational}, $\Sigma$ is closed with genus one. Moreover, $\Sigma$ has entropy less than $2$ since the shrinking doughnut $\mathbb T$ in \cite{DruganNguyenShrinkingDonut} does, and by  \cite{Liu2016_index_shrinker}, $\Sigma$ has index at least $7$. Therefore, Theorem \ref{thm_eternal}  still holds with $\mathbb T$ replaced by $\Sigma$: The exact same proof will work. As a result, we obtain a genus one shrinker with entropy strictly lower than $\Sigma$. However, the self-shrinkers with the three lowest entropy are the plane, the sphere, and the cylinder (\cite{CIMW13_EntropyMinmzer, BernsteinWang17_small_entropy}). Contradiction arises.

	\bibliographystyle{alpha}
	\bibliography{GMT}
\end{document}